\documentclass[aop,preprint]{imsart}

\RequirePackage{amsthm,amsmath,amsfonts,amssymb}
\RequirePackage[numbers]{natbib}
\RequirePackage{graphicx}

\usepackage{verbatim}
\usepackage{dsfont,bm}
\usepackage{xcolor}
\usepackage{enumitem}

\startlocaldefs
\numberwithin{equation}{section}
\numberwithin{figure}{section}

\newtheorem{theorem}{Theorem}[section]
\newtheorem{corollary}[theorem]{Corollary}
\newtheorem{lemma}[theorem]{Lemma}
\newtheorem{proposition}[theorem]{Proposition}
\newtheorem{conjecture}[theorem]{Conjecture}
\theoremstyle{remark}
\newtheorem{definition}[theorem]{Definition}
\newtheorem{assumption}[theorem]{Assumption}
\newtheorem{remark}[theorem]{Remark}

\newtheorem*{notation}{Notation}
\newcommand{\ver}{Version of September 10, 2020} 

\definecolor{darkgreen}{rgb}{0,0.75,0}
\definecolor{darkred}{rgb}{0.75,0,0}
\definecolor{darkmagenta}{rgb}{0.5,0,0.5}
\usepackage[colorlinks,citecolor=darkgreen,linkcolor=darkred,urlcolor=darkmagenta]{hyperref}

\newcommand{\mr}[1]{{\tt \href{http://www.ams.org/mathscinet-getitem?mr=#1}{MR#1}}}
\renewcommand{\arxiv}[1]{{\tt \href{http://arxiv.org/abs/#1}{arXiv:#1}}}
\newcommand{\one}{\mathds{1}}
\newcommand{\loc}{\operatorname{loc}}

\newcommand{\diam}{\mathop{\operatorname{diam}}}
\newcommand{\supp}{\mathop{\operatorname{supp}}}
\newcommand{\on}[1]{\operatorname{#1}}
\endlocaldefs

\begin{document}

\begin{frontmatter}
\title{On singularity of energy measures for symmetric diffusions with full off-diagonal heat kernel estimates}
\runtitle{On singularity of energy measures for symmetric diffusions}

\begin{aug}
\author[A]{\fnms{Naotaka} \snm{Kajino}\ead[label=e1]{nkajino@math.kobe-u.ac.jp}}
\and
\author[B]{\fnms{Mathav} \snm{Murugan}\ead[label=e2]{mathav@math.ubc.ca}}
\runauthor{N.\ Kajino and M.\ Murugan}
\address[A]{Department of Mathematics, Graduate School of Science, Kobe University, \printead{e1}}
\address[B]{Department of Mathematics, University of British Columbia, \printead{e2}}
\end{aug}

\begin{abstract}
We show that for a strongly local, regular symmetric Dirichlet form over
a complete, locally compact geodesic metric space, full off-diagonal
heat kernel estimates with walk dimension strictly larger than two
(\emph{sub-Gaussian} estimates) imply the singularity of the energy measures
with respect to the symmetric measure, verifying a conjecture by M.\ T.\ Barlow in
[Contemp.\ Math., vol.\ 338, 2003, pp.\ 11--40].
We also prove that in the contrary case of walk dimension two, i.e., where
full off-diagonal \emph{Gaussian} estimates of the heat kernel hold,
the symmetric measure and the energy measures are mutually absolutely continuous
in the sense that a Borel subset of the state space has measure zero for the symmetric
measure if and only if it has measure zero for the energy measures of all functions
in the domain of the Dirichlet form.
\end{abstract}



\end{frontmatter}


\section{Introduction} \label{sec:intro}

It is an established result in the field of analysis on fractals that, on a large class of
typical fractal spaces, there exists a nice diffusion process $\{X_{t}\}_{t\geq 0}$ which
is symmetric with respect to some canonical measure $m$ and exhibits strong sub-diffusive
behavior in the sense that its transition density (heat kernel) $p_{t}(x,y)$ satisfies
the following \textbf{sub-Gaussian estimates}:
\begin{equation}\label{e:HKEbeta}
\begin{split}
\frac{c_{1}}{m(B(x,t^{1/\beta}))} \exp\biggl( - c_{2} \Bigl( \frac{d(x,y)^{\beta}}{t} &\Bigr)^{\frac{1}{\beta-1}}\biggr)
	\leq p_{t}(x,y)\\
&\leq\frac{c_{3}}{m(B(x,t^{1/\beta}))} \exp\biggl( - c_{4} \Bigl( \frac{d(x,y)^{\beta}}{t} \Bigr)^{\frac{1}{\beta-1}}\biggr)
\end{split}
\end{equation}
for all points $x,y$ and all $t>0$, where $c_{1},c_{2},c_{3},c_{4}>0$ are some constants,
$d$ is a natural geodesic metric on the space, $B(x,r)$ denotes the open ball of radius
$r$ centered at $x$, and $\beta\geq 2$ is a characteristic of the diffusion called the
\textbf{walk dimension}. This result was obtained first for the Sierpi\'{n}ski gasket in
\cite{BP}, then for nested fractals in \cite{Kum}, for affine nested fractals in \cite{FHK}
and for Sierpi\'{n}ski carpets in \cite{BB89,BB92,BB99} (see also \cite{KZ,BBK,BBKT}),
and in most of (essentially all) the known examples it turned out that $\beta>2$;
see, e.g., Proposition \ref{p:sisg-case-sing} below and \cite{Kaj20a} for an elementary
proof of $\beta>2$ for Sierpi\'{n}ski gaskets and carpets, respectively.
Therefore \eqref{e:HKEbeta} implies in particular
that a typical distance the diffusion travels by time $t$ is of order $t^{1/\beta}$,
which is in sharp contrast with the order $t^{1/2}$ of such a distance for the
Brownian motion and uniformly elliptic diffusions on Euclidean spaces and
Riemannian manifolds, where \eqref{e:HKEbeta} with $\beta=2$,
the usual \textbf{Gaussian estimates}, are known to hold widely;
see, e.g., \cite{Stu95a,Stu96,SC,Gri} and references therein.

The main concern of this paper is singularity and absolute continuity of the energy
measures associated with a general $m$-symmetric diffusion $\{X_{t}\}_{t\geq 0}$ satisfying
\eqref{e:HKEbeta} for some $\beta\geq 2$, on a locally compact separable metric measure
space $(X,d,m)$. Under the standard assumption of the regularity of the Dirichlet form
$(\mathcal{E},\mathcal{F})$ of $\{X_{t}\}_{t\geq 0}$, the \textbf{energy measure}
of a function $f\in\mathcal{F}\cap L^{\infty}(X,m)$ is defined as
the unique Borel measure $\Gamma(f,f)$ on $X$ such that
\begin{align}\label{e:EnergyMeas-intro}
\int_{X} g \, d\Gamma(f,f) &= \mathcal{E}(f,fg)-\frac{1}{2}\mathcal{E}(f^{2},g)\\
	&=\lim_{t\downarrow 0}\frac{1}{2t}\int_{X}g(x)\mathbb{E}_{x}\bigl[|f(X_{t})-f(X_{0})|^{2}\bigr]\,dm(x) \nonumber
\end{align}
for any $g\in\mathcal{F}\cap\mathcal{C}_{\mathrm{c}}(X)$, where a quasi-continuous $m$-version
of $f$ is used for defining $\{f(X_{t})\}_{t\geq 0}$ and $\mathcal{C}_{\mathrm{c}}(X)$
denotes the set of $\mathbb{R}$-valued continuous functions on $X$ with compact supports.
Then the approximation of $f$ by $\{(-n)\vee(f\wedge n)\}_{n=1}^{\infty}$ defines
$\Gamma(f,f)$ also for general $f\in\mathcal{F}$. In probabilistic terms, let
\begin{equation} \label{e:F-decomp}
f(X_{t})-f(X_{0}) = M^{[f]}_{t} + N^{[f]}_{t}, \qquad t\geq 0,
\end{equation}
be the \emph{Fukushima decomposition}, an extension of It\^{o}'s formula and the
semimartingale decomposition in the framework of regular symmetric Dirichlet forms,
of $\{f(X_{t})-f(X_{0})\}_{t\geq 0}$ into the sum of the martingale part
$M^{[f]}=\{M^{[f]}_{t}\}_{t\geq 0}$ and the zero-energy part
$N^{[f]}=\{N^{[f]}_{t}\}_{t\geq 0}$ (see \cite[Theorem 5.2.2]{FOT}). Then
$\Gamma(f,f)$ arises as the Revuz measure of the quadratic variation process
$\langle M^{[f]}\rangle=\{\langle M^{[f]}\rangle_{t}\}_{t\geq 0}$ of $M^{[f]}$
(see \cite[Theorems 5.1.3 and 5.2.3]{FOT}). Therefore the question of whether $\Gamma(f,f)$
is singular with respect to $m$ could be considered as an analytical counterpart of
that of whether $\langle M^{[f]}\rangle=\{\langle M^{[f]}\rangle_{t}\}_{t\geq 0}$
is singular as a function in $t\in[0,\infty)$, although the actual relation between
these two questions is unclear. A better-founded probabilistic implication, due
to \cite[Proposition 12]{HN}, of the singularity of $\Gamma(f,f)$ with respect to $m$
for all $f\in\mathcal{F}$ is that there is no representation, in a certain stochastic
sense, of the diffusion $\{X_{t}\}_{t\geq 0}$ in terms of a Brownian motion on
$\mathbb{R}^{k}$ for any $k\in\mathbb{N}$; see \cite[Section 4]{HN} for details.

When $(\mathcal{E},\mathcal{F})$ is given, on the basis of some differential structure
on $X$, by $\mathcal{E}(f,g)=\int_{X}\langle \nabla f, \nabla g\rangle_{x}\,dm(x)$
for some first-order differential operator $\nabla$ satisfying the usual Leibniz rule
and some (measurable) Riemannian metric $\langle\cdot,\cdot\rangle_{x}$,
the right-hand side of \eqref{e:EnergyMeas-intro} is easily seen to be equal to
$\int_{X}g(x)\langle \nabla f,\nabla f\rangle_{x}\,dm(x)$ and hence
$d\Gamma(f,f)(x)=\langle \nabla f,\nabla f\rangle_{x}\,dm(x)$. In particular,
$\Gamma(f,f)$ is absolutely continuous with respect to the symmetric measure $m$.

On the other hand, diffusions on self-similar fractals are known to exhibit completely
different behavior. For a class of self-similar fractals including the Sierpi\'{n}ski gasket,
Kusuoka showed in \cite{Kus89} that the energy measures are \emph{singular} with respect
to the symmetric measure, which in the case of the Sierpi\'{n}ski gasket is the standard
$\log_{2}3$-dimensional Hausdorff measure. Later in \cite{Kus93} he extended
this result to the case of the Brownian motion on a class of nested fractals, and
Ben-Bassat, Strichartz and Teplyaev \cite{BST} obtained similar results for a class
of self-similar Dirichlet forms on post-critically finite self-similar fractals
under simpler assumptions and with a shorter proof.

The best result known so far in this direction is due to Hino \cite{Hin05}. There he proved
that for a general self-similar Dirichlet form on a self-similar set,
\emph{including the case of the Brownian motion on Sierpi\'{n}ski carpets},
the following dichotomy holds for each self-similar (Bernoulli) measure $\mu$
(including the symmetric measure $m$):
\begin{itemize}[label={\quad either \textup{(i)}},align=left,leftmargin=*]
\item[either \textup{(i)}]$\mu=\Gamma(h,h)$ for some $h\in\mathcal{F}$ that is harmonic on the complement
	of the canonical ``boundary'' of the self-similar set,
\item[or \textup{(ii)}]$\Gamma(f,f)$ is singular with respect to $\mu$ for any $f\in\mathcal{F}$.
\end{itemize}
It was also proved in \cite{Hin05} that the lower inequality in \eqref{e:HKEbeta} for
the heat kernel $p_{t}(x,y)$ with $\beta>2$, which is known to hold in particular for
Sierpi\'{n}ski carpets by the results in \cite{BB92,BB99,Kaj20a} (see also \cite{KZ,BBK,BBKT}),
excludes the possibility of case (i) for $\mu=m$ and thus implies
the singularity of $\Gamma(f,f)$ with respect to the symmetric measure $m$ for any
$f\in\mathcal{F}$. This is the only existing result proving the singularity of the
energy measures for diffusions on \emph{infinitely ramified} self-similar fractals
like Sierpi\'{n}ski carpets. The reader is also referred to \cite{HN} for simple
geometric conditions which exclude case (i) of the above dichotomy in the setting
of post-critically finite self-similar sets.

All these results on the singularity of the energy measures heavily relied on the
exact self-similarity of the state space and the Dirichlet form. In reality, however,
even without the self-similarity the anomalous space-time scaling relation exhibited
by the terms $t^{1/\beta}$ and $d(x,y)^{\beta}/t$ in \eqref{e:HKEbeta} should still
imply singular behavior of the sample paths of the quadratic variation
$\langle M^{[f]}\rangle$ of the martingale part $M^{[f]}$ in \eqref{e:F-decomp}.
Therefore it is natural to conjecture, as Barlow did in \cite[Section 5, Remarks, 5.-1.]{Bar03},
that the heat kernel estimates \eqref{e:HKEbeta} with $\beta>2$ should imply the
singularity of the energy measures with respect to the symmetric measure $m$. The
first half of our main result (Theorem \ref{t:main}-(a)) verifies this conjecture in
the completely general framework of a strongly local, regular symmetric Dirichlet form
$(\mathcal{E},\mathcal{F})$ over a complete, locally compact separable metric measure space
$(X,d,m)$ satisfying a certain geodesic-like property called the \textbf{chain condition}
(see Definition \ref{d:chain}-(a)) and the \textbf{volume doubling property}
\begin{equation}\label{e:VD-intro}
m(B(x,2r)) \leq C_D m(B(x,r)), \qquad (x,r)\in X\times(0,\infty).
\end{equation}

Note here that the chain condition is necessary for making the strict inequality
$\beta>2$ for the exponent $\beta$ in \eqref{e:HKEbeta} meaningful.
Indeed, by \cite[Corollary 1.8 (or Theorem 2.11)]{Mur} and \cite[Proof of Theorem 6.5]{GT},
under the general framework mentioned above, \eqref{e:HKEbeta} is equivalent to
the conjunction of the chain condition, \eqref{e:VD-intro}, the upper inequality
in \eqref{e:HKEbeta} and the so-called \textbf{near-diagonal lower estimate}
\begin{equation}\label{e:NLHKEbeta}
p_{t}(x,y) \geq \frac{c_{1}}{m(B(x,t^{1/\beta}))}
	\qquad \textrm{for all $x,y\in X$ with $d(x,y) \leq \delta t^{1/\beta}$}
\end{equation}
for some constants $c_{1},\delta>0$. Then by \cite[Theorem 7.4]{GT}, this latter set
of conditions with the chain condition dropped is characterized, under the additional
assumption that $X$ is non-compact, by the conjunction of \eqref{e:VD-intro},
the scale-invariant elliptic Harnack inequality and the mean exit time estimate
\begin{equation}\label{e:exit-time}
c_{5}r^{\beta} \leq \mathbb{E}_{x}[\tau_{B(x,r)}] \leq c_{6}r^{\beta},
	\qquad(x,r)\in X\times(0,\infty),
\end{equation}
where $\tau_{B(x,r)}:=\inf\{t\in[0,\infty)\mid X_{t}\not\in B(x,r)\}$ ($\inf\emptyset:=\infty$).
Since the last characterization is preserved under the change of the metric from $d$ to
$d^{\alpha}$ for any $\alpha\in(0,1)$ with the price of replacing $\beta$ by $\beta/\alpha$,
it follows that we would be able to realize an arbitrarily large value of $\beta\geq 2$
by suitable changes of metrics if we did not assume the chain condition.

To complement the above result for the case of $\beta>2$, as the second half of our main
result (Theorem \ref{t:main}-(b)) we also prove that \eqref{e:HKEbeta} with $\beta=2$
implies the ``mutual absolute continuity'' between the symmetric measure $m$ and
the energy measures $\Gamma(f,f)$, i.e., that for each Borel subset $A$ of the state
space $X$, $m(A)=0$ if and only if $\Gamma(f,f)(A)=0$ for any $f\in\mathcal{F}$.
In the context of studying \eqref{e:HKEbeta} with $\beta=2$ (Gaussian estimates),
it is customary to \emph{assume} from the beginning of the analysis that $\Gamma(f,f)$ is
absolutely continuous with respect to $m$ for a large class of $f\in\mathcal{F}$,
whereas we \emph{deduce from} \eqref{e:HKEbeta} with $\beta=2$ this absolute
continuity for all $f\in\mathcal{F}$ as part of Theorem \ref{t:main}-(b);
see Remark \ref{rmk:dom} for some related results.

In fact, we state and prove our result in a slightly wider framework allowing a general
space-time scaling function $\Psi$ instead of considering just $\Psi(r)=r^{\beta}$. This
generalization enables us to conclude the singularity of the energy measures for the canonical
Dirichlet forms on (spatially homogeneous) \emph{scale irregular Sierpi\'{n}ski gaskets}
studied in \cite{Ham92,BH,Ham00}, which are not exactly self-similar and hence are outside
of the frameworks of the preceding works \cite{Kus89,Kus93,BST,Hin05,HN}. See also
\cite[Chapter 24]{Kig12} for a discussion of these examples and Section \ref{s:sisg}
below for the proof that Theorem \ref{t:main}-(a) applies to (at least some of) them.

This paper is organized as follows. In Section \ref{sec:framework-result} we introduce the
framework in detail and give the precise statement of our main result (Theorem \ref{t:main}).
Then its first half on the singularity of the energy measures (Theorem \ref{t:main}-(a))
is proved in Section \ref{s:sing} and its second half on the absolute continuity
(Theorem \ref{t:main}-(b)) is proved in Section \ref{s:ac}.
An application of Theorem \ref{t:main}-(a) to some scale irregular
Sierpi\'{n}ski gaskets is presented in Section \ref{s:sisg}.
In Appendix \ref{sec:appendix}, for the reader's convenience we give complete proofs
of a couple of miscellaneous facts utilized in the proof of Theorem \ref{t:main}-(a).

\begin{notation}
Throughout this paper, we use the following notation and conventions.
\begin{itemize}[label=\textup{(b)},align=left,leftmargin=*]
\item[\textup{(a)}]The symbols $\subset$ and $\supset$ for set inclusion
	\emph{allow} the case of the equality.
\item[\textup{(b)}]$\mathbb{N}:=\{n\in\mathbb{Z}\mid n>0\}$, i.e., $0\not\in\mathbb{N}$.
\item[\textup{(c)}]The cardinality (the number of elements) of a set $A$ is denoted by $\#A$.
\item[\textup{(d)}]We set $\infty^{-1}:=0$. We write
	$a\vee b:=\max\{a,b\}$, $a\wedge b:=\min\{a,b\}$, $a^{+}:=a\vee 0$ and
	$a^{-}:=-(a\wedge 0)$ for $a,b\in[-\infty,\infty]$, and
	we use the same notation also for $[-\infty,\infty]$-valued functions and
	equivalence classes of them. All numerical functions in this paper are assumed
	to be $[-\infty,\infty]$-valued.
\item[\textup{(e)}]Let $X$ be a non-empty set. We define $\one_{A}=\one_{A}^{X}\in\mathbb{R}^{X}$ for $A\subset X$ by
	$\one_{A}(x):=\one_{A}^{X}(x):=\bigl\{\begin{smallmatrix}1 & \textrm{if $x\in A$,}\\ 0 & \textrm{if $x\not\in A$,}\end{smallmatrix}$
	and set $\|f\|_{\sup}:=\|f\|_{\sup,X}:=\sup_{x\in X}|f(x)|$ for $f:X\to[-\infty,\infty]$.
\item[\textup{(f)}]Let $X$ be a topological space. We set
	$\mathcal{C}(X):=\{f\mid\textrm{$f:X\to\mathbb{R}$, $f$ is continuous}\}$ and
	$\mathcal{C}_{\mathrm{c}}(X):=\{f\in\mathcal{C}(X)\mid\textrm{$X\setminus f^{-1}(0)$ has compact closure in $X$}\}$.
\item[\textup{(g)}]Let $(X,\mathcal{B})$ be a measurable space and let $\mu,\nu$ be $\sigma$-finite
	measures on $(X,\mathcal{B})$. We write $\nu \ll \mu$ and $\nu \perp \mu$ to mean that
	$\nu$ is absolutely continuous and singular, respectively, with respect to $\mu$.
	We set $\mu|_{A}:=\mu|_{\mathcal{B}|_{A}}$ for $A\in\mathcal{B}$,
	where $\mathcal{B}|_{A}:=\{B\cap A\mid B\in\mathcal{B}\}$.
\end{itemize}
\end{notation}

\section{Framework and the main result}\label{sec:framework-result}

In this section, we introduce the framework of this paper and state our main result.
After introducing the framework of a strongly local regular Dirichlet space and
the associated energy measures in Subsection \ref{ssec:MMD-EnergyMeas}, we give
in Subsection \ref{ssec:HKE-VDPICS} the precise formulation of the off-diagonal
heat kernel estimates and an equivalent condition for the estimates which
is convenient for the proof of the main result. Then we give the statement of our main
theorem (Theorem \ref{t:main}) in Subsection \ref{ssec:result} and outline its proof
in Subsection \ref{ssec:outline}.

\subsection{Metric measure Dirichlet space and energy measure}\label{ssec:MMD-EnergyMeas}

Throughout this paper, we consider a \emph{complete}, locally compact separable metric
space $(X,d)$, equipped with a Radon measure $m$ with full support, i.e., a Borel measure
$m$ on $X$ which is finite on any compact subset of $X$ and strictly positive on any
non-empty open subset of $X$, and we always assume $\#X\geq 2$ to exclude the trivial
case of $\#X=1$. Such a triple $(X,d,m)$ is referred to as a \emph{metric measure space}.
We set $B(x,r):=\{y\in X\mid d(x,y)<r\}$ for $(x,r) \in X\times(0,\infty)$ and
$\diam(X,d):=\sup_{x,y\in X}d(x,y)$; note that $\#X\geq 2$ is equivalent to $\diam(X,d)\in(0,\infty]$.

Furthermore let $(\mathcal{E},\mathcal{F})$ be a \emph{symmetric Dirichlet form} on $L^{2}(X,m)$;
by definition, $\mathcal{F}$ is a dense linear subspace of $L^{2}(X,m)$, and
$\mathcal{E}:\mathcal{F}\times\mathcal{F}\to\mathbb{R}$
is a non-negative definite symmetric bilinear form which is \emph{closed}
($\mathcal{F}$ is a Hilbert space under the inner product $\mathcal{E}_{1}:= \mathcal{E}+ \langle \cdot,\cdot \rangle_{L^{2}(X,m)}$)
and \emph{Markovian} ($f^{+}\wedge 1\in\mathcal{F}$ and $\mathcal{E}(f^{+}\wedge 1,f^{+}\wedge 1)\leq \mathcal{E}(f,f)$ for any $f\in\mathcal{F}$).
Recall that $(\mathcal{E},\mathcal{F})$ is called \emph{regular} if
$\mathcal{F}\cap\mathcal{C}_{\mathrm{c}}(X)$ is dense both in $(\mathcal{F},\mathcal{E}_{1})$
and in $(\mathcal{C}_{\mathrm{c}}(X),\|\cdot\|_{\mathrm{sup}})$, and that
$(\mathcal{E},\mathcal{F})$ is called \emph{strongly local} if $\mathcal{E}(f,g)=0$
for any $f,g\in\mathcal{F}$ with $\supp_{m}[f]$, $\supp_{m}[g]$ compact and
$\supp_{m}[f-a\one_{X}]\cap\supp_{m}[g]=\emptyset$ for some $a\in\mathbb{R}$. Here
for a Borel measurable function $f:X\to[-\infty,\infty]$ or an
$m$-equivalence class $f$ of such functions, $\supp_{m}[f]$ denotes the support of the measure $|f|\,dm$,
i.e., the smallest closed subset $F$ of $X$ with $\int_{X\setminus F}|f|\,dm=0$,
which exists since $X$ has a countable open base for its topology; note that
$\supp_{m}[f]$ coincides with the closure of $X\setminus f^{-1}(0)$ in $X$ if $f$ is continuous.
The pair $(X,d,m,\mathcal{E},\mathcal{F})$ of a metric measure space $(X,d,m)$ and a strongly local,
regular symmetric Dirichlet form $(\mathcal{E},\mathcal{F})$ on $L^{2}(X,m)$ is termed
a \emph{metric measure Dirichlet space}, or a \emph{MMD space} in abbreviation.
We refer to \cite{FOT,CF} for details of the theory of symmetric Dirichlet forms.

The central object of the study of this paper is the energy measures associated
with a MMD space, which are defined as follows. Note that $fg\in\mathcal{F}$
for any $f,g\in\mathcal{F}\cap L^{\infty}(X,m)$ by \cite[Theorem 1.4.2-(ii)]{FOT}
and that $\{(-n)\vee(f\wedge n)\}_{n=1}^{\infty}\subset\mathcal{F}$ and
$\lim_{n\to\infty}(-n)\vee(f\wedge n)=f$ in norm in $(\mathcal{F},\mathcal{E}_{1})$
by \cite[Theorem 1.4.2-(iii)]{FOT}.

\begin{definition}[{\cite[(3.2.13), (3.2.14) and (3.2.15)]{FOT}}]\label{d:EnergyMeas}
Let $(X,d,m,\mathcal{E},\mathcal{F})$ be a MMD space.
The \textbf{energy measure} $\Gamma(f,f)$ of $f\in\mathcal{F}$
associated with $(X,d,m,\mathcal{E},\mathcal{F})$ is defined,
first for $f\in\mathcal{F}\cap L^{\infty}(X,m)$ as the unique ($[0,\infty]$-valued)
Borel measure on $X$ such that
\begin{equation}\label{e:EnergyMeas}
\int_{X} g \, d\Gamma(f,f)= \mathcal{E}(f,fg)-\frac{1}{2}\mathcal{E}(f^{2},g) \qquad \textrm{for all $g \in \mathcal{F}\cap\mathcal{C}_{\mathrm{c}}(X)$,}
\end{equation}
and then by
$\Gamma(f,f)(A):=\lim_{n\to\infty}\Gamma\bigl((-n)\vee(f\wedge n),(-n)\vee(f\wedge n)\bigr)(A)$
for each Borel subset $A$ of $X$ for general $f\in\mathcal{F}$. We also define
the \textbf{mutual energy measure} $\Gamma(f,g)$ of $f,g\in\mathcal{F}$
as the Borel signed measure on $X$ given by
$\Gamma(f,g):=\frac{1}{2}\bigl(\Gamma(f+g,f+g)-\Gamma(f,f)-\Gamma(g,g)\bigr)$,
so that $\Gamma(\mspace{-0.8mu}\cdot\mspace{-0.8mu},\mspace{-0.8mu}\cdot\mspace{-0.8mu})$
is bilinear and symmetric and satisfies the Cauchy--Schwarz inequality:
\begin{align}\label{e:bilinear-energy}
\Gamma(&af+bg,af+bg)=a^{2}\Gamma(f,f)+2ab\Gamma(f,g)+b^{2}\Gamma(g,g), \qquad a,b \in \mathbb{R},\\
|\Gamma&(f,g)(B)|^{2} \leq \Gamma(f,f)(B)\Gamma(g,g)(B) \qquad \textrm{for all Borel subsets $B$ of $X$.}
\label{e:Cauchy-Schwarz-energy}
\end{align}
Note that by \cite[Lemma 3.2.3]{FOT} and the strong locality of $(\mathcal{E},\mathcal{F})$,
\begin{equation}\label{e:EnergyMeasTotalMass}
\Gamma(f,g)(X)=\mathcal{E}(f,g) \qquad \textrm{for all $f,g \in \mathcal{F}$.}
\end{equation}
\end{definition}

\subsection{Off-diagonal heat kernel estimates and equivalent condition}\label{ssec:HKE-VDPICS}

The most general form of the off-diagonal heat kernel estimates, which we are introducing
in Definition \ref{d:HKE} below, involves a homeomorphism $\Psi:[0,\infty) \to [0,\infty)$
representing the scaling relation between time and space variables:

\begin{assumption} \label{a:reg}
Throughout this paper, we fix a homeomorphism $\Psi:[0,\infty) \to [0,\infty)$ such that
\begin{equation} \label{e:reg}
C_{\Psi}^{-1} \biggl( \frac{R}{r} \biggr)^{\beta_{0}} \leq \frac{\Psi(R)}{\Psi(r)}
	\leq C_{\Psi} \biggl( \frac{R}{r} \biggr)^{\beta_{1}}
\end{equation}
for all $0 < r \leq R$ for some constants $1 < \beta_{0} \leq \beta_{1}$ and $C_{\Psi} \geq 1$.
\end{assumption}

The following condition is standard and often treated as part of the standing
assumptions in the context of heat kernel estimates on general MMD spaces.

\begin{definition}[\hypertarget{vd}{$\on{VD}$}]
Let $(X,d,m)$ be a metric measure space. We say that $(X,d,m)$ satisfies the
\textbf{volume doubling property} \hyperlink{vd}{$\on{VD}$},
if there exists a constant $C_{D}>1$ such that for all $x \in X$ and all $r>0$,
\begin{equation} \tag*{\hyperlink{vd}{$\on{VD}$}}
m(B(x,2r)) \leq C_{D} m(B(x,r)).
\end{equation}
Note that if $(X,d,m)$ satisfies \hyperlink{vd}{$\on{VD}$}, then $B(x,r)$ is relatively
compact (i.e., has compact closure) in $X$ for all $(x,r)\in X\times(0,\infty)$
by virtue of the completeness of $(X,d)$.
\end{definition}

\begin{definition}[\hypertarget{hke}{$\on{HKE}(\Psi)$}]\label{d:HKE}
Let $(X,d,m,\mathcal{E},\mathcal{F})$ be a MMD space, and let $\{P_{t}\}_{t>0}$
denote its associated Markov semigroup. A family $\{p_{t}\}_{t>0}$ of
$[0,\infty]$-valued Borel measurable functions on $X \times X$ is called the
\emph{heat kernel} of $(X,d,m,\mathcal{E},\mathcal{F})$, if $p_{t}$ is the integral kernel
of the operator $P_t$ for any $t>0$, that is, for any $t > 0$ and for any $f \in L^{2}(X,m)$,
\begin{equation*}
P_{t} f(x) = \int_{X} p_{t}(x,y) f (y)\, dm (y) \qquad \textrm{for $m$-a.e.\ $x \in X$.}
\end{equation*}
We say that $(X,d,m,\mathcal{E},\mathcal{F})$ satisfies the \textbf{heat kernel estimates}
\hyperlink{hke}{$\on{HKE}(\Psi)$}, if its heat kernel $\{p_{t}\}_{t>0}$ exists and
there exist $C_{1},c_{1},c_{2},c_{3},\delta\in(0,\infty)$ such that for each $t>0$,
\begin{align}\label{e:uhke}
p_{t}(x,y) &\leq \frac{C_{1}}{m\bigl(B(x,\Psi^{-1}(t))\bigr)} \exp \bigl( -c_{1} \Phi( c_{2}d(x,y), t ) \bigr)
	\qquad \textrm{for $m$-a.e.\ $x,y \in X$,}\\
p_{t}(x,y) &\ge \frac{c_{3}}{m\bigl(B(x,\Psi^{-1}(t))\bigr)}
	\qquad \textrm{for $m$-a.e.\ $x,y\in X$ with $d(x,y) \leq \delta\Psi^{-1}(t)$,}
\label{e:nlhke}
\end{align}
where
\begin{equation} \label{e:defPhi}
\Phi(R,t) := \Phi_{\Psi}(R,t) := \sup_{r>0} \biggl(\frac{R}{r}-\frac{t}{\Psi(r)}\biggr),
	\qquad (R,t)\in[0,\infty)\times(0,\infty).
\end{equation}
\end{definition}

\begin{remark}\label{rmk:defPhi}
\begin{itemize}[label=\textup{(b)},align=left,leftmargin=*]
\item[\textup{(a)}]It easily follows from \eqref{e:reg} that \eqref{e:defPhi} defines
	a lower semi-continuous function $\Phi=\Phi_{\Psi}:[0,\infty)\times(0,\infty)\to[0,\infty)$
	such that for any $R,t\in(0,\infty)$, $\Phi(0,t)=0$, $\Phi(\cdot,t)$ is strictly
	increasing and $\Phi(R,\cdot)$ is strictly decreasing.
\item[\textup{(b)}]If $\beta>1$ and $\Psi$ is given by $\Psi(r)=r^{\beta}$, then an elementary differential
	calculus shows that $\Phi(R,t)=(\beta-1)\beta^{-\frac{\beta}{\beta-1}}(R^{\beta}/t)^{\frac{\beta}{\beta-1}}$
	for any $(R,t)\in[0,\infty)\times(0,\infty)$, in which case the right-hand side of
	\eqref{e:uhke} coincides with that of \eqref{e:HKEbeta}.
\item[\textup{(c)}]If a MMD space $(X,d,m,\mathcal{E},\mathcal{F})$ satisfies \hyperlink{vd}{$\on{VD}$} and
	\hyperlink{hke}{$\on{HKE}(\Psi)$}, then there exists a version of the heat kernel $p_{t}(x,y)$
	which is continuous in $(t,x,y)\in(0,\infty)\times X\times X$; see, e.g., \cite[Theorem 3.1]{BGK}.
\item[\textup{(d)}]If a MMD space $(X,d,m,\mathcal{E},\mathcal{F})$ satisfies the chain condition
	(see Definition \ref{d:chain}-(a) below) in addition to \hyperlink{vd}{$\on{VD}$}
	and \hyperlink{hke}{$\on{HKE}(\Psi)$}, then \eqref{e:nlhke} can be strengthened to
	a lower bound of the same form as \eqref{e:uhke} valid for $m$-a.e.\ $x,y \in X$;
	see, e.g., \cite[Proof of Theorem 6.5]{GT}. Note that this global lower bound
	implies \eqref{e:nlhke} since $\Phi(c_{2}d(x,y),t)$ is less than
	some constant as long as $d(x,y) \le \delta\Psi^{-1}(t)$ by \cite[(5.13)]{GK}.
\end{itemize}
\end{remark}

In fact, \hyperlink{hke}{$\on{HKE}(\Psi)$} itself is not very convenient for analyzing
the energy measures, and there is a characterization of \hyperlink{hke}{$\on{HKE}(\Psi)$} by
the conjunction of two functional inequalities which are more suitable for our purpose, defined as follows.

\begin{definition}[\hypertarget{pi}{$\operatorname{PI}(\Psi)$} and \hypertarget{cs}{$\operatorname{CS}(\Psi)$}]\label{d:PI-CS}
Let $(X,d,m,\mathcal{E},\mathcal{F})$ be a MMD space.
\begin{itemize}[label=\textup{(b)},align=left,leftmargin=*]
\item[\textup{(a)}]We say that $(X,d,m,\mathcal{E},\mathcal{F})$ satisfies the
	\textbf{Poincar\'e inequality} \hyperlink{pi}{$\operatorname{PI}(\Psi)$},
	if there exist constants $C_{P}>0$ and $A\geq 1$ such that 
	for all $(x,r)\in X\times(0,\infty)$ and all $f \in \mathcal{F}$,
	\begin{equation} \tag*{\hyperlink{pi}{$\operatorname{PI}(\Psi)$}}
	\int_{B(x,r)} |f - f_{B(x,r)}|^{2} \,dm \leq C_{P} \Psi(r) \int_{B(x,Ar)}d\Gamma(f,f),
	\end{equation}
	where $f_{B(x,r)}:= m(B(x,r))^{-1} \int_{B(x,r)} f\, dm$.
\item[\textup{(b)}]For open subsets $U,V$ of $X$ with $\overline{U} \subset V$,
	where $\overline{U}$ denotes the closure of $U$ in $X$, we say that
	a function $\varphi \in \mathcal{F}$ is a \emph{cutoff function} for $U \subset V$
	if $0 \leq \varphi \leq 1$ $m$-a.e., $\varphi=1$ $m$-a.e.\ on
	a neighbourhood of $\overline{U}$ and $\supp_{m}[\varphi] \subset V$.
	Then we say that $(X,d,m,\mathcal{E},\mathcal{F})$ satisfies the
	\textbf{cutoff Sobolev inequality} \hyperlink{cs}{$\operatorname{CS}(\Psi)$},
	if there exists $C_{S}>0$ such that the following holds:
	for each $x \in X$ and each $R,r>0$ there exists a cutoff function $\varphi \in \mathcal{F}$
	for $B(x,R) \subset B(x,R+r)$ such that for all $f \in \mathcal{F}$,
	\begin{equation} \tag*{\hyperlink{cs}{$\operatorname{CS}(\Psi)$}}
	\int_{X} f^{2}\, d\Gamma(\varphi,\varphi)
		\leq \frac{1}{8} \int_{B(x,R+r) \setminus B(x,R)} \varphi^{2} \, d\Gamma(f,f)
			+ \frac{C_{S}}{\Psi(r)} \int_{B(x,R+r) \setminus B(x,R)} f^{2}\,dm.
	\end{equation}
	Here and in what follows, we always consider a quasi-continuous $m$-version
	of $f\in\mathcal{F}$, which exists by \cite[Theorem 2.1.3]{FOT} and is unique
	$\mathcal{E}$-q.e.\ (i.e., up to sets of capacity zero) by \cite[Lemma 2.1.4]{FOT},
	so that the values of $f$ are uniquely determined $\Gamma(g,g)$-a.e.\ for each $g\in\mathcal{F}$
	since $\Gamma(g,g)(N)=0$ for any Borel subset $N$ of $X$ of capacity zero by
	\cite[Lemma 3.2.4]{FOT}; see \cite[Section 2.1]{FOT} for the definitions of
	the capacity and the quasi-continuity of functions with respect to a
	regular symmetric Dirichlet form.
\end{itemize}
\end{definition}

\begin{remark}\label{rmk:CS}
The specific constant $\frac{1}{8}$ in the right-hand side of \hyperlink{cs}{$\operatorname{CS}(\Psi)$}
is chosen for the sake of convenience in its use; see, e.g., the proof of Lemma \ref{l:reverse-pi} below.
There is no harm in making this choice because \hyperlink{cs}{$\operatorname{CS}(\Psi)$} is equivalent
to the same condition with $\frac{1}{8}$ replaced by arbitrary $C_{S}'>0$ under
Assumption \ref{a:reg} for $\Psi$ by \cite[Lemma 5.1]{AB}, whose proof is easily
seen to be valid without assuming $\diam(X,d)=\infty$ or \hyperlink{vd}{$\on{VD}$}.
\end{remark}

\begin{theorem}[\cite{BB04,BBK,AB,GHL}; see also {\cite[Theorem 3.2]{Lie}}]\label{t:HKE-VDPICS}
If a MMD space $(X,d,m,\mathcal{E},\mathcal{F})$ satisfies \hyperlink{vd}{$\on{VD}$} and
\hyperlink{hke}{$\on{HKE}(\Psi)$}, then it also satisfies \hyperlink{pi}{$\on{PI}(\Psi)$}
and \hyperlink{cs}{$\on{CS}(\Psi)$} and $(X,d)$ is connected.
\end{theorem}

\begin{remark}\label{rmk:HKE-VDPICS}
\begin{itemize}[label=\textup{(b)},align=left,leftmargin=*]
\item[\textup{(a)}]The converse of Theorem \ref{t:HKE-VDPICS} has been proved in
	\cite[Theorem 1.2]{GHL} under the additional assumption that $(X,d)$ is non-compact:
	\begin{equation*}
	\begin{minipage}{330pt}
	\emph{If a MMD space $(X,d,m,\mathcal{E},\mathcal{F})$ satisfies \hyperlink{vd}{$\on{VD}$},
	\hyperlink{pi}{$\on{PI}(\Psi)$} and \hyperlink{cs}{$\on{CS}(\Psi)$} and $(X,d)$ is connected and
	non-compact, then $(X,d,m,\mathcal{E},\mathcal{F})$ also satisfies \hyperlink{hke}{$\on{HKE}(\Psi)$}}.
	\end{minipage}
	\end{equation*}
	This converse implication should be true even without assuming the non-compactness of $(X,d)$,
	because \cite[Theorem 4.2]{GT} seems to be the only relevant result in \cite{GT,GH,GHL}
	requiring seriously the non-compactness but a suitable modification of it can be
	in fact proved by using \cite[Theorem 6.2]{GK} also in the case where $(X,d)$ is compact.
	Since the converse would not increase the applicability of our main theorem
	(Theorem \ref{t:main}), which assumes \hyperlink{pi}{$\on{PI}(\Psi)$} and
	\hyperlink{cs}{$\on{CS}(\Psi)$} rather than \hyperlink{hke}{$\on{HKE}(\Psi)$},
	we refrain from going into further details of its validity.
\item[\textup{(b)}]There is a (minor but) non-trivial technical gap in the proofs
	of the implication from \hyperlink{vd}{$\on{VD}$} and \hyperlink{hke}{$\on{HKE}(\Psi)$}
	to \hyperlink{pi}{$\on{PI}(\Psi)$} presented in \cite{GHL,Lie}. Specifically,
	both of their proofs utilize the Neumann and Dirichlet heat semigroups
	$\{P^{\mathrm{N},B}_{t}\}_{t>0}$ and $\{P^{\mathrm{D},B}_{t}\}_{t>0}$,
	respectively, on a given ball $B:=B(x,r)$ and the inequality
	\begin{equation}\label{e:proof-PI-gap}
	\int_{B}P^{\mathrm{N},B}_{t}\bigl(|f-g(y)|^{2}\bigr)(y)\,dm(y)
		\geq\int_{B}P^{\mathrm{D},B}_{t}\bigl(|f-g(y)|^{2}\bigr)(y)\,dm(y)
	\end{equation}
	for $f,g\in L^{2}(B,m|_{B})$ and $t\in(0,\infty)$, but
	\emph{the expressions $P^{\mathrm{N},B}_{t}\bigl(|f-g(y)|^{2}\bigr)(y)$ and
	$P^{\mathrm{D},B}_{t}\bigl(|f-g(y)|^{2}\bigr)(y)$ in \eqref{e:proof-PI-gap}
	do not make literal sense}. While the latter can be still interpreted as
	representing $\int_{B}p^{\mathrm{D},B}_{t}(y,z)|f(z)-g(y)|^{2}\,dm(z)$
	with $\{p^{\mathrm{D},B}_{t}\}_{t>0}$ denoting the heat kernel of
	$\{P^{\mathrm{D},B}_{t}\}_{t>0}$, whose existence is easily implied by
	\hyperlink{hke}{$\on{HKE}(\Psi)$}, the former does not allow even this way of
	interpretation because the heat kernel of $\{P^{\mathrm{N},B}_{t}\}_{t>0}$ might
	not exist. In fact, \eqref{e:proof-PI-gap} should rather be interpreted as
	\begin{equation}\label{e:proof-PI-correct}
	\begin{split}
	\int_{B}\bigl(P^{\mathrm{N},B}_{t}(f^{2})-2gP^{\mathrm{N},B}_{t}f&+g^{2}P^{\mathrm{N},B}_{t}\one_{B}\bigr)\,dm\\
		&\geq\int_{B}\bigl(P^{\mathrm{D},B}_{t}(f^{2})-2gP^{\mathrm{D},B}_{t}f+g^{2}P^{\mathrm{D},B}_{t}\one_{B}\bigr)\,dm,
	\end{split}
	\end{equation}
	which follows from the observation that, if additionally $g$ is a simple function on $B$, then
	\begin{align*}
	P^{\mathrm{N},B}_{t}(f^{2})-2gP^{\mathrm{N},B}_{t}f+g^{2}P^{\mathrm{N},B}_{t}\one_{B}
		&=\sum_{a\in g(B)}\one_{g^{-1}(a)}P^{\mathrm{N},B}_{t}\bigl(|f-a\one_{B}|^{2}\bigr)\\
	\geq\sum_{a\in g(B)}\one_{g^{-1}(a)}P^{\mathrm{D},B}_{t}\bigl(|f-a\one_{B}|^{2}\bigr)
		&=P^{\mathrm{D},B}_{t}(f^{2})-2gP^{\mathrm{D},B}_{t}f+g^{2}P^{\mathrm{D},B}_{t}\one_{B}
		\quad \textrm{$m|_{B}$-a.e.}
	\end{align*}
	Now the proofs of \hyperlink{pi}{$\on{PI}(\Psi)$} in \cite[Proof of Theorem 1.2]{GHL}
	and \cite[Proof of Theorem 3.2]{Lie} can be easily justified by replacing
	\eqref{e:proof-PI-gap} with \eqref{e:proof-PI-correct} in their arguments.
	\end{itemize}
\end{remark}

\subsection{Statement of the main result} \label{ssec:result}

The statement of our main result (Theorem \ref{t:main} below) requires some more definitions.
First, the following conditions on the metric are crucial for Theorem \ref{t:main},
especially for its first half on the singularity of the energy measures.

\begin{definition} \label{d:chain}
Let $(X,d)$ be a metric space.
\begin{itemize}[label=\textup{(b)},align=left,leftmargin=*]
\item[\textup{(a)}]For $\varepsilon>0$ and $x,y\in X$, we say that a sequence $\{x_{i}\}_{i=0}^{N}$
	of points in $X$ is an \emph{$\varepsilon$-chain} in $(X,d)$ from $x$ to $y$ if 
	\begin{equation*}
	N \in \mathbb{N}, \quad x_{0}=x, \quad x_{N}=y \quad \textrm{and} \quad d(x_{i},x_{i+1}) < \varepsilon \quad \textrm{for all $i\in\{0,1,\ldots,N-1\}$.}
	\end{equation*}
	Then for $\varepsilon>0$ and $x,y \in X$, define (with the convention that $\inf\emptyset:=\infty$)
	\begin{equation} \label{e:depsilon}
	d_{\varepsilon}(x,y) := \inf\Biggl\{\sum_{i=0}^{N-1} d(x_{i},x_{i+1})
		\Biggm|\textrm{$\{x_{i}\}_{i=0}^{N}$ is an $\varepsilon$-chain in $(X,d)$ from $x$ to $y$}
		\Biggr\}.
	\end{equation}
	We say that $(X,d)$ satisfies the \textbf{chain condition} if there exists $C \geq 1$ such that
	\begin{equation} \label{e:chain}
	d_{\varepsilon}(x,y) \leq C d(x,y)\qquad \textrm{for all $\varepsilon>0$ and all $x,y \in X$.}
	\end{equation}
\item[\textup{(b)}]We say that $(X,d)$ (or $d$) is \textbf{geodesic} if for any $x,y\in X$
	there exists $\gamma:[0,1]\to X$ such that $\gamma(0)=x$, $\gamma(1)=y$
	and $d(\gamma(s),\gamma(t))=|s-t|d(x,y)$ for any $s,t\in[0,1]$.
\end{itemize}
In fact, under the assumption that $B(x,r)$ is relatively compact in $X$ for all
$(x,r)\in X\times(0,\infty)$, $(X,d)$ satisfies the chain condition if and only if
$d$ is bi-Lipschitz equivalent to a geodesic metric $\rho$ on $X$;
see Proposition \ref{p:cc} in Appendix \ref{ssec:chain-geodesic}.
\end{definition}

The following definition is standard in studying Gaussian heat kernel estimates, i.e.,
\eqref{e:uhke} with $\Psi(r)=r^{2}$ and the matching lower estimate of $p_t(x,y)$.

\begin{definition}\label{d:dint}
Let $(X,d,m,\mathcal{E},\mathcal{F})$ be a MMD space. We define its
\textbf{intrinsic metric} $d_{\on{int}}:X\times X\to[0,\infty]$ by
\begin{equation}\label{e:dint}
d_{\on{int}}(x,y) := \sup \bigl\{f(x) -f(y) \bigm|
	\textrm{$f \in \mathcal{F}_{\on{loc}} \cap \mathcal{C}(X)$, $\Gamma(f,f) \leq m$} \bigr\},
\end{equation}
where
\begin{equation}\label{e:Floc}
\mathcal{F}_{\loc} := \Biggl\{ f \Biggm|
	\begin{minipage}{255pt}
	$f$ is an $m$-equivalence class of $\mathbb{R}$-valued Borel measurable functions
	on $X$ such that $f \one_{V} = f^{\#} \one_{V}$ $m$-a.e.\ for some $f^{\#}\in\mathcal{F}$
	for each relatively compact open subset $V$ of $X$
	\end{minipage}
	\Biggr\}
\end{equation}
and the energy measure $\Gamma(f,f)$ of $f\in\mathcal{F}_{\loc}$ associated with
$(X,d,m,\mathcal{E},\mathcal{F})$ is defined as the unique Borel measure on $X$
such that $\Gamma(f,f)(A)=\Gamma(f^{\#},f^{\#})(A)$ for any relatively compact
Borel subset $A$ of $X$ and any $V,f^{\#}$ as in \eqref{e:Floc} with $A\subset V$;
note that $\Gamma(f^{\#},f^{\#})(A)$ is independent of a particular choice of such $V,f^{\#}$
by \eqref{e:Cauchy-Schwarz-energy}, \eqref{e:bilinear-energy} and \cite[Corollary 3.2.1]{FOT}.
\end{definition}

In the literature on Gaussian heat kernel estimates it is customary to assume that the
intrinsic metric $d_{\on{int}}$ is a complete metric on $X$ compatible with the original
topology of $X$, in which case it sounds natural in view of \eqref{e:dint} to guess that
the symmetric measure $m$ and the family of energy measures $\Gamma(f,f)$ should be
``mutually absolutely continuous''. The following definition due to \cite{Hin10}
rigorously formulates the notion of such a measure.

\begin{definition}[{\cite[Definition 2.1]{Hin10}}]\label{d:minimal-energy-dominant}
Let $(X,d,m,\mathcal{E},\mathcal{F})$ be a MMD space. A $\sigma$-finite Borel measure
$\nu$ on $X$ is called a \textbf{minimal energy-dominant measure}
of $(\mathcal{E},\mathcal{F})$ if the following two conditions are satisfied:
\begin{itemize}[label=\textup{(ii)},align=left,leftmargin=*]
\item[(i)](Domination) For every $f \in \mathcal{F}$, $\Gamma(f,f) \ll \nu$.
\item[(ii)](Minimality) If another $\sigma$-finite Borel measure $\nu'$
	on $X$ satisfies condition (i) with $\nu$ replaced
	by $\nu'$, then $\nu \ll \nu'$.
\end{itemize}
Note that by \cite[Lemmas 2.2, 2.3 and 2.4]{Hin10}, a minimal energy-dominant measure of
$(\mathcal{E},\mathcal{F})$ always exists and is precisely a $\sigma$-finite Borel measure
$\nu$ on $X$ such that for each Borel subset $A$ of $X$, $\nu(A)=0$ if and only if
$\Gamma(f,f)(A)=0$ for all $f\in\mathcal{F}$.
\end{definition}

Now we can state the main theorem of this paper, which asserts that the conjunction of
\hyperlink{vd}{$\on{VD}$}, \hyperlink{pi}{$\operatorname{PI}(\Psi)$} and
\hyperlink{cs}{$\operatorname{CS}(\Psi)$} implies the singularity and
the absolute continuity of the energy measures, if $\Psi(r)$ decays as $r\downarrow 0$
sufficiently faster than $r^{2}$ and at most as fast as $r^{2}$, respectively.
We also describe what the intrinsic metric $d_{\on{int}}$ looks like in each case.
Remember that the assumption of \hyperlink{vd}{$\on{VD}$}, \hyperlink{pi}{$\operatorname{PI}(\Psi)$}
and \hyperlink{cs}{$\operatorname{CS}(\Psi)$} in the following theorem can be replaced
with that of \hyperlink{vd}{$\on{VD}$} and \hyperlink{hke}{$\on{HKE}(\Psi)$} by virtue
of Theorem \ref{t:HKE-VDPICS} and that $\diam(X,d)\in(0,\infty]$ by $\#X\geq 2$.

\begin{theorem} \label{t:main}
Let $(X,d,m,\mathcal{E},\mathcal{F})$ be a MMD space satisfying \hyperlink{vd}{$\on{VD}$},
\hyperlink{pi}{$\on{PI}(\Psi)$} and \hyperlink{cs}{$\on{CS}(\Psi)$}.
\begin{itemize}[label=\textup{(b)},align=left,leftmargin=*]
\item[\textup{(a)}]\textup{(Singularity)} If $(X,d)$ satisfies the chain condition and
	\begin{equation}\label{e:case-sing}
	\liminf_{\lambda \to \infty}\liminf_{r \downarrow 0} \frac{\lambda^2 \Psi(r/\lambda)}{\Psi(r)} =0,
	\end{equation}
	then $\Gamma(f,f) \perp m$ for all $f \in \mathcal{F}$.
	In this case, the intrinsic metric $d_{\on{int}}$ is identically zero.
\item[\textup{(b)}]\textup{(Absolute continuity)} If
	\begin{equation}\label{e:case-ac}
	\limsup_{r \downarrow 0} \frac{\Psi (r)}{r^{2}} > 0,
	\end{equation}
	then $m$ is a minimal energy-dominant measure of $(\mathcal{E},\mathcal{F})$,
	and in particular $\Gamma(f,f) \ll m$ for all $f \in \mathcal{F}$. In this case,
	the intrinsic metric $d_{\on{int}}$ is a geodesic metric on $X$ and
	there exist $r_{1},r_{2}\in(0,\diam(X,d))$ and $C_{1},C_{2}\geq 1$ such that
	\begin{align}\label{e:ge2}
	C_{1}^{-1} r^{2} \leq \Psi(r) &\leq C_{1} r^{2} &&\textrm{for all $r\in(0,r_{1})$,}\\
	C_{2}^{-1} d(x,y) \leq d_{\on{int}}(x,y) &\leq C_{2} d(x,y)
		&&\textrm{for all $x,y \in X$ with $d(x,y) \wedge d_{\on{int}}(x,y) < r_{2}$.}
	\label{e:bi-Lipschitz-d-dint}
	\end{align}
	Furthermore if additionally $(X,d)$ satisfies the chain condition, then $d_{\on{int}}$
	is bi-Lipschitz equivalent to $d$, that is, \eqref{e:bi-Lipschitz-d-dint} with
	$r_{2}=\infty$ holds for some $C_{2}\geq 1$.
\end{itemize}
\end{theorem}

\begin{remark} \label{rmk:main}
If $\Psi(r)=r^{\beta}$ for some $\beta>1$, then \eqref{e:case-sing}
is equivalent to $\beta>2$ and \eqref{e:case-ac} is equivalent to $\beta \leq 2$.
For general $\Psi$, however, the conditions \eqref{e:case-sing} and \eqref{e:case-ac}
are not complementary to each other since there are examples of $\Psi$ satisfying
Assumption \ref{a:reg} but not either of \eqref{e:case-sing} and \eqref{e:case-ac}; indeed,
for each $k\in\mathbb{N}$, the homeomorphism $\Psi_{k}:[0,\infty)\to[0,\infty)$ given by
\begin{equation} \label{e:r2logn}
\Psi_{k}(r) := r^{2} \eta_{0}^{\circ k}(r\wedge 1), \qquad \textrm{where} \quad
	\eta_{0}(r) := \frac{1}{\log(e-1+r^{-1})} \quad (\eta_{0}(0):=0)
\end{equation}
and $\eta_{0}^{\circ k}$ denotes the $k$-fold composition of $\eta_{0}:[0,\infty)\to[0,\infty)$,
is such an example. In fact, for a large class of such $\Psi$ including $\Psi_{k}$ as in
\eqref{e:r2logn}, it is possible to construct a MMD space which is equipped with a geodesic
metric and satisfies \hyperlink{vd}{$\on{VD}$} and \hyperlink{hke}{$\on{HKE}(\Psi)$},
by considering a class of fractals obtained by modifying the construction of the
scale irregular Sierpi\'{n}ski gaskets in Section \ref{s:sisg} below in the
following manner suggested by Barlow in \cite{Bar19}:

For each $\bm{l}=(l_{n})_{n=1}^{\infty}\in(\mathbb{N}\setminus\{1,2,3,4\})^{\mathbb{N}}$,
we define the \emph{$2$-dimensional level-$\bm{l}$ thin scale irregular Sierpi\'{n}ski gasket}
$\hat{K}^{\bm{l}}$ by \eqref{e:sisg} with $N=2$ and with $S_{l}$ in Section \ref{s:sisg} replaced by
\begin{equation*}
\hat{S}_{l}:=\bigl\{(i_{1},i_{2})\in(\mathbb{N}\cup\{0\})^{2}
	\bigm|\textrm{$i_{1}+i_{2}\leq l-1$, $i_{1}i_{2}(l-1-i_{1}-i_{2})=0$}\bigr\},
\end{equation*}
i.e., with the ways of cell subdivision in Section \ref{s:sisg} modified so as to
keep only the cells along the boundary of the triangle at each subdivision step.
Then we can define in exactly the same way as Section \ref{s:sisg} a canonical MMD space
$(\hat{K}^{\bm{l}},\hat{d}_{\bm{l}},\hat{m}_{\bm{l}},\hat{\mathcal{E}}^{\bm{l}},\hat{\mathcal{F}}_{\bm{l}})$
over $\hat{K}^{\bm{l}}$ with the metric $\hat{d}_{\bm{l}}$ geodesic, and furthermore it can
be shown, \emph{regardless of the possible unboundedness of $\bm{l}=(l_{n})_{n=1}^{\infty}$},
to satisfy \hyperlink{vd}{$\on{VD}$}, \hyperlink{hke}{$\on{HKE}(\hat{\Psi}_{\bm{l}})$}
for a homeomorphism $\hat{\Psi}_{\bm{l}}:[0,\infty)\to[0,\infty)$ defined explicitly in terms of
$\bm{l}$, and $\hat{\Gamma}_{\bm{l}}(f,f) \perp \hat{m}_{\bm{l}}$ for all $f\in\hat{\mathcal{F}}_{\bm{l}}$
for its associated energy measures $\hat{\Gamma}_{\bm{l}}(\cdot,\cdot)$. Now it is possible to prove that
for each homeomorphism $\eta:[0,1]\to[0,1]$ satisfying $\eta(0)=0$ and the (seemingly mild) condition that
\begin{equation} \label{e:eta}
\sum_{n=0}^{\infty} \frac{ \eta^{-1}(2^{-n-1}) }{ \eta^{-1}(2^{-n}) } < \infty,
\end{equation}
which holds, e.g., for $\eta_{0}^{\circ k}$ as in \eqref{e:r2logn} for any $k\in\mathbb{N}$,
there exist $\bm{l}_{\eta}\in(\mathbb{N}\setminus\{1,2,3,4\})^{\mathbb{N}}$ and $C_{\eta}\geq 1$ such that
\begin{equation} \label{e:Psi-eta}
C_{\eta}^{-1} \hat{\Psi}_{\bm{l}_{\eta}}(r) \leq \Psi_{\eta}(r):=r^{2}\eta(r\wedge 1)
	\leq C_{\eta} \hat{\Psi}_{\bm{l}_{\eta}}(r) \qquad \textrm{for any $r\in[0,\infty)$.}
\end{equation}
The details of the results stated in this paragraph will appear in the forthcoming paper \cite{Kaj20b}.

Since the decay rate of $\hat{\Psi}_{\bm{l}}(r)$ as $r\downarrow 0$ can be made arbitrarily
close to that of $r^{2}$ by taking suitable $\bm{l}\in(\mathbb{N}\setminus\{1,2,3,4\})^{\mathbb{N}}$,
e.g., $\bm{l}_{\eta}$ as in \eqref{e:Psi-eta} with $\eta=\eta_{0}^{\circ k}$
for arbitrarily large $k\in\mathbb{N}$ yet the associated MMD space
$(\hat{K}^{\bm{l}},\hat{d}_{\bm{l}},\hat{m}_{\bm{l}},\hat{\mathcal{E}}^{\bm{l}},\hat{\mathcal{F}}_{\bm{l}})$
still satisfies $\hat{\Gamma}_{\bm{l}}(f,f) \perp \hat{m}_{\bm{l}}$ for all $f\in\hat{\mathcal{F}}_{\bm{l}}$,
it is natural to expect that we would always have $\Gamma(f,f) \perp m$ for all $f \in \mathcal{F}$
under the assumptions of Theorem \ref{t:main} unless \eqref{e:case-ac} holds.
Namely, we have the following conjecture:
\end{remark}

\begin{conjecture}[Energy measure singularity dichotomy] \label{conj:main}
Let $(X,d,m,\mathcal{E},\mathcal{F})$ be a MMD space satisfying \hyperlink{vd}{$\on{VD}$},
\hyperlink{pi}{$\on{PI}(\Psi)$} and \hyperlink{cs}{$\on{CS}(\Psi)$}, and assume further
that $(X,d)$ satisfies the chain condition and that
\begin{equation}\label{e:case-nonGauss}
\lim_{r \downarrow 0} \frac{\Psi (r)}{r^{2}} = 0.
\end{equation}
Then $\Gamma(f,f) \perp m$ for all $f \in \mathcal{F}$.
\end{conjecture}

\subsection{Outline of the proof}\label{ssec:outline}
The proofs of Theorem \ref{t:main}-(a) and Theorem \ref{t:main}-(b)
are completed in Sections \ref{s:sing} and \ref{s:ac}, respectively.

In Section \ref{s:sing}, we reduce the proof of Theorem \ref{t:main}-(a) to the case of
harmonic functions by approximating an arbitrary function in $\mathcal{F}$ by ``piecewise
harmonic functions'' --- see Propositions \ref{p:approx-harmonic} and \ref{p:approx}.
The proof proceeds by contradiction. If the energy measure $\Gamma(h,h)$ of a harmonic
function $h$ has a non-trivial absolutely continuous part with respect to the symmetric
measure $m$, then by Lebesgue's differentiation theorem we can approximate $\Gamma(h,h)$
by a constant multiple of $m$ locally at sufficiently many scales --- see Lemma \ref{l:zoom}.
Then we can estimate the variances of $h$
on small balls from above by using \hyperlink{pi}{$\operatorname{PI}(\Psi)$} and
from below by \hyperlink{cs}{$\operatorname{CS}(\Psi)$} and the harmonicity of $h$
--- see \eqref{e:hm3} and \eqref{e:hm4}. The conjunction of these upper and lower
bounds contradicts the assumption \eqref{e:case-sing} on $\Psi$.

In Section \ref{s:ac}, we prove Theorem \ref{t:main}-(b). We first deduce \eqref{e:ge2}
from the assumption \eqref{e:case-ac} and a recent result \cite[Corollary 1.10]{Mur}
by the second-named author (Lemma \ref{l:ge2}). We next show that for small enough $r$,
the function $(r-d(x,\cdot))^{+}$ belongs to $\mathcal{F}$ and has energy measure
absolutely continuous with respect to the symmetric measure $m$ (Lemma \ref{l:dist}).
Then we approximate any function in $\mathcal{F}$ by using combinations of functions of
the form $(r-d(x,\cdot))^{+}$ --- see Lemma \ref{l:lip} and Proposition \ref{p:dom}. The
minimality of $m$ follows from \hyperlink{pi}{$\operatorname{PI}(\Psi)$} and Lemma \ref{l:dist}
(Proposition \ref{p:min}), the finiteness of the intrisic metric $d_{\on{int}}$ from
\eqref{e:bi-Lipschitz-d-dint} and \cite[Lemma 2.2]{Mur} (Proposition \ref{p:bilip}),
and we finally conclude the bi-Lipschitz equivalence of $d_{\on{int}}$ to $d$
(Proposition \ref{p:bilip}) by combining \eqref{e:bi-Lipschitz-d-dint}, the chain condition
for $(X,d)$ and the geodesic property of $d_{\on{int}}$ proved in \cite[Theorem 1]{Stu95b}.

\begin{notation}
In the following, we will use the notation $A \lesssim B$ for quantities $A$ and $B$
to indicate the existence of an implicit constant $C > 0$ depending on some
inessential parameters such that $A \le CB$.
\end{notation}

\section{Singularity} \label{s:sing}
In this section, we give the proof of Theorem \ref{t:main}-(a), i.e., the singularity
of the energy measures under the assumption \eqref{e:case-sing}. We start with a lemma
describing the local behavior of a Radon measure in relation to another with \hyperlink{vd}{$\on{VD}$}.
\begin{lemma} \label{l:zoom}
Let $(X,d,m)$ be a metric measure space satisfying \hyperlink{vd}{$\on{VD}$}, and
let $\nu$ be a Radon measure on $X$, i.e., a Borel measure on $X$ which is finite on
any compact subset of $X$. Let $\nu=\nu_{a} + \nu_{s}$ denote the Lebesgue decomposition
of $\nu$ with respect to $m$, where $\nu_{a} \ll m$ and $\nu_{s} \perp m$.
Let $\delta_{0} \in (0,1)$. Then for $m$-a.e.\ $x \in \bigl\{z\in X \bigm| \frac{d\nu_{a}}{dm}(z)>0\bigr\}$,
there exists $r_{0}=r_{0}(x,\delta_{0})>0$ such that for every $r\in(0,r_{0})$,
every $\delta \in [\delta_{0},1]$ and every $y \in B(x,r)$,
\begin{equation} \label{e:rn}
\frac{1}{2} \frac{d\nu_{a}}{dm}(x)
	\leq \frac{\nu(B(y,\delta r))}{m(B(y,\delta r))}
	\leq 2 \frac{d\nu_{a}}{dm}(x).
\end{equation}
\end{lemma}

\begin{proof}
Let $f:=\frac{d\nu_{a}}{dm}$ denote the Radon--Nikodym derivative.
By \hyperlink{vd}{$\on{VD}$} and \cite[(2.8)]{Hei}, 
\begin{equation} \label{e:rn1}
\lim_{r \downarrow 0} \frac{1}{ m(B(x,r))} \int_{B(x,r)} |f(z)-f(x)|\, dm(z)=0
\end{equation}
for $m$-a.e.\ $x \in X$. Also there exists $C_{1}>0$ (which depends only
on the constant $C_{D}$ in \hyperlink{vd}{$\on{VD}$} and $\delta_{0}$) such that
for all $x \in X$, $r>0$, $\delta \in [\delta_{0},1]$ and $y\in B(x,r)$ we have
\begin{align} \label{e:rn2}
&\frac{\bigl|\nu_a(B(y,\delta r)) - f(x) m(B(y,\delta r))\bigr|}{m(B(y,\delta r))} \\
&= \frac{\bigl|\int_{B(y,\delta r)} (f(z)-f(x))\,dm(z)\bigr|}{m(B(y,\delta r))} \nonumber  \\
&\leq \frac{ \int_{B(x,2r)} |f(z)-f(x)|\,dm(z)}{m(B(y,\delta r))} \qquad \textrm{(by $B(y,\delta  r) \subset B(y,r) \subset B(x,2r)$)} \nonumber \\
&\leq \frac{ \int_{B(x,2r)} |f(z)-f(x)|\,dm(z)}{m(B(x,2r))} \frac{m(B(y,3r))}{m(B(y,\delta_0 r))} \qquad \textrm{(by $B(x,2r) \subset B(y,3r)$ and $\delta \ge \delta_0$)} \nonumber \\
&\leq C_{1} \frac{ \int_{B(x,2r)} |f(z)-f(x)|\,dm(z)}{m(B(x,2r))} \qquad \textrm{(by \hyperlink{vd}{$\on{VD}$}).} \nonumber
\end{align}
Using \eqref{e:rn1} and \eqref{e:rn2}, we obtain the following: for every $x \in X$ satisfying
$\frac{d\nu_{a}}{dm}(x)>0$ and \eqref{e:rn1}, there exists $r_{1}=r_{1}(x,\delta_{0})>0$
such that for all $r \in (0, r_{1})$, $\delta \in [\delta_{0},1]$ and $y \in B(x,r)$,
\begin{equation} \label{e:rn3}
\frac{1}{2} \frac{d\nu_{a}}{dm}(x)
	\leq \frac{\nu_{a}(B(y, \delta r))}{m(B(y,\delta r))}
	\leq \frac{3}{2} \frac{d\nu_{a}}{dm}(x).
\end{equation}
On the other hand, by Proposition \ref{p:rd} in Appendix \ref{ssec:Lebesgue-diff-sing}
(see also \cite[Theorem 7.13]{Rud}),
\begin{equation} \label{e:rn4}
\lim_{r \downarrow 0} \frac{\nu_{s}(B(x,r))}{m(B(x,r))} = 0
\end{equation}
for $m$-a.e.\ $x \in X$. By using \hyperlink{vd}{$\on{VD}$} as in \eqref{e:rn2} above,
we obtain the following: for every $x \in X$ satisfying
$\frac{d\nu_{a}}{dm}(x)>0$ and \eqref{e:rn4}, there exists $r_{2}=r_{2}(x,\delta_{0})>0$
such that for all $r \in (0,r_{2})$, $\delta \in [\delta_{0},1]$ and $y \in B(x,r)$,
\begin{equation} \label{e:rn5}
\frac{\nu_{s}(B(y,\delta r))}{m(B(y,\delta r))}
	\leq C_{1} \frac{\nu_{s}(B(x,2r))}{m(B(x,2r))}
	\leq \frac{1}{2}\frac{d\nu_{a}}{dm}(x).
\end{equation}
Combining \eqref{e:rn1}, \eqref{e:rn3}, \eqref{e:rn4} and \eqref{e:rn5},
we get the desired conclusion with $r_{0} = r_{1} \wedge r_{2}$.
\end{proof}

We first prove the singularity of the energy measures of harmonic functions,
which are defined in the present framework as follows.

\begin{definition}\label{d:harmonic}
Let $(X,d,m,\mathcal{E},\mathcal{F})$ be a MMD space. A function $h \in \mathcal{F}$
is said to be \emph{$\mathcal{E}$-harmonic} on an open subset $U$ of $X$, if
\begin{equation}\label{e:harmonic}
\mathcal{E}(h,f)=0 \qquad
	\begin{minipage}{260pt}
		for all $f\in\mathcal{F}\cap\mathcal{C}_{\mathrm{c}}(X)$ with $\supp_{m}[f]\subset U$,
		or equivalently, for all $f\in\mathcal{F}_{U}:=\{g\in\mathcal{F}\mid\textrm{$g=0$ $\mathcal{E}$-q.e.\ on $X\setminus U$}\}$,
	\end{minipage}
\end{equation}
where the equivalence of the two definitions follows from \cite[Corollary 2.3.1]{FOT}.
\end{definition}

The following reverse Poincar\'e inequality is an easy consequence of
\hyperlink{cs}{$\operatorname{CS}(\Psi)$}.

\begin{lemma}[Reverse Poincar\'e inequality] \label{l:reverse-pi}
Let $(X,d,m,\mathcal{E},\mathcal{F})$ be a MMD space satisfying
\hyperlink{cs}{$\operatorname{CS}(\Psi)$}, and let $C_{S}$ denote the
constant in \hyperlink{cs}{$\operatorname{CS}(\Psi)$}. Then for any
$(x,r)\in X\times (0,\infty)$, any $a\in\mathbb{R}$ and any function
$h\in\mathcal{F}\cap L^{\infty}(X,m)$ that is $\mathcal{E}$-harmonic on $B(x,2r)$,
\begin{equation}\label{e:reverse-pi}
\int_{B(x,r)} d\Gamma(h,h) \leq \frac{8 C_{S}}{\Psi(r)} \int_{B(x,2r)\setminus B(x,r)} |h-a|^{2} \, dm.
\end{equation}
\end{lemma}

\begin{proof}
Let $(x,r)\in X\times (0,\infty)$ and let $h\in\mathcal{F}\cap L^{\infty}(X,m)$ be
$\mathcal{E}$-harmonic on $B(x,2r)$. By the regularity of $(\mathcal{E},\mathcal{F})$ and
\cite[Exercise 1.4.1]{FOT} we can take $g\in\mathcal{F}\cap\mathcal{C}_{\mathrm{c}}(X)$
such that $g=1$ on $B(x,2r)$, and then $g$ is $\mathcal{E}$-harmonic on $B(x,2r)$ and
$\Gamma(g,g)(B(x,2r))=0$ by the strong locality of $(\mathcal{E},\mathcal{F})$ and
\cite[Corollary 3.2.1]{FOT} (or \cite[Theorem 4.3.8]{CF}), whence
$\Gamma(h,h)|_{B(x,2r)}=\Gamma(h-ag,h-ag)|_{B(x,2r)}$ for any $a\in\mathbb{R}$
by \eqref{e:bilinear-energy} and \eqref{e:Cauchy-Schwarz-energy}.
Therefore \eqref{e:reverse-pi} for general $a\in\mathbb{R}$ follows from
\eqref{e:reverse-pi} for $a=0$ by considering $h-ag$ instead of $h$.

Let $\varphi$ be a cutoff function for $B(x,r) \subset B(x,2r)$
from \hyperlink{cs}{$\operatorname{CS}(\Psi)$}. Then since
$h,\varphi\in\mathcal{F}\cap L^{\infty}(X,m)$, $h$ is $\mathcal{E}$-harmonic
on $B(x,2r)$ and $h\varphi^{2}=0$ $\mathcal{E}$-q.e.\ on $X\setminus B(x,2r)$
by $\supp_{m}[\varphi]\subset B(x,2r)$ and \cite[Lemma 2.1.4]{FOT}, we have
\begin{align}\label{e:reverse-pi-proof}
0 &= \mathcal{E}(h,h \varphi^{2}) = \Gamma(h,h\varphi^{2})(X) \qquad \textrm{(by \eqref{e:harmonic} and \eqref{e:EnergyMeasTotalMass})} \\
	&= \int_{X} \varphi^{2}\, d\Gamma(h,h) + 2 \int_{X} \varphi h \, d\Gamma(h, \varphi) \qquad \textrm{(by \cite[Lemma 3.2.5]{FOT})} \nonumber \\
	&\geq \int_{X} \varphi^{2}\, d\Gamma(h,h) - 2 \sqrt{ \int_{X} \varphi^{2} \, d\Gamma(h, h) \int_{X} h^{2} \, d\Gamma(\varphi, \varphi) } \nonumber \\[-8pt]
	&\mspace{350mu} \textrm{(by \cite[Proof of Lemma 5.6.1]{FOT})} \nonumber \\
	&\geq \int_{X} \varphi^{2}\, d\Gamma(h,h) - \frac{1}{2} \int_{X} \varphi^{2} \, d\Gamma(h, h) - 2 \int_{X} h^{2} \, d\Gamma(\varphi,\varphi) \nonumber \\
	&\geq \frac{1}{4} \int_{X} \varphi^{2}\, d\Gamma(h,h) - \frac{2 C_{S}}{\Psi(r)} \int_{B(x,2r) \setminus B(x,r)} h^{2} \, dm \qquad \textrm{(by \hyperlink{cs}{$\operatorname{CS}(\Psi)$}).} \nonumber
\end{align}
Noting that $\varphi=1$ $\mathcal{E}$-q.e.\ on $B(x,r)$ by \cite[Lemma 2.1.4]{FOT}
and hence that $\varphi=1$ $\Gamma(h,h)$-a.e.\ on $B(x,r)$ by \cite[Lemma 3.2.4]{FOT},
from \eqref{e:reverse-pi-proof} we now obtain
\begin{equation*}
\int_{B(x,r)} d\Gamma(h,h)\leq \int_{X} \varphi^{2}\, d\Gamma(h,h)
	\leq \frac{8 C_{S}}{\Psi(r)} \int_{B(x,2r)\setminus B(x,r)} h^{2}\, dm,
\end{equation*}
proving \eqref{e:reverse-pi} for $a=0$.
\end{proof}

Proposition \ref{p:harmonic} below establishes the singularity of the energy measures
of $\mathcal{E}$-harmonic functions. For our convenience, we introduce the notion
of an $\varepsilon$-net in a metric space as follows.

\begin{definition} \label{d:net}
Let $(X,d)$ be a metric space and let $\varepsilon>0$.
A subset $N$ of $X$ is called an \emph{$\varepsilon$-net} in $(X,d)$
if the following two conditions are satisfied:
\begin{itemize}[label=\textup{(ii)},align=left,leftmargin=*]
\item[\textup{(i)}](Separation) $N$ is \emph{$\varepsilon$-separated} in $(X,d)$, i.e., $d(x,y) \geq \varepsilon$ for any $x,y \in N$ with $x \not= y$.
\item[\textup{(ii)}](Maximality) If $N\subset M\subset X$ and $M$ is $\varepsilon$-separated in $(X,d)$, then $M=N$.
\end{itemize}
It is elementary to see that an $\varepsilon$-net in $(X,d)$ exists
if $B(x,r)$ is totally bounded in $(X,d)$ for any $(x,r)\in X\times(0,\infty)$
and that any $\varepsilon$-net in $(X,d)$ is finite if $(X,d)$ is totally bounded.
\end{definition}

\begin{proposition}\label{p:harmonic}
Let $(X,d,m,\mathcal{E},\mathcal{F})$ be a MMD space satisfying \hyperlink{vd}{$\on{VD}$},
\hyperlink{pi}{$\on{PI}(\Psi)$} and \hyperlink{cs}{$\on{CS}(\Psi)$}, and assume
further that $d$ is geodesic and that $\Psi$ satisfies \eqref{e:case-sing}.
Let $U$ be an open subset of $X$ and let $h\in\mathcal{F}\cap L^{\infty}(X,m)$
be $\mathcal{E}$-harmonic on $U$. Then $\Gamma(h,h)|_{U} \perp m|_{U}$.
\end{proposition}

\begin{proof}
Assume to the contrary that the conclusion $\Gamma(h,h)|_{U} \perp m|_{U}$ fails. Let
$A\geq 1$ denote the constant in \hyperlink{pi}{$\on{PI}(\Psi)$} and let $\lambda>4A$.
By Lemma \ref{l:zoom}, and by replacing $h$ with $\alpha h$ for some suitable
$\alpha \in (0,\infty)$ if necessary, there exist $x \in U$ and $r_{x,\lambda} >0$
with $B(x,r_{x,\lambda}) \subset U$ such that for all $r \in (0,r_{x,\lambda})$,
$\delta \in [\lambda^{-1},1]$ and $y \in B(x,r)$,
\begin{equation} \label{e:hm0}
\frac{1}{2} \leq \frac{\Gamma(h,h)(B(y, \delta r))}{m(B(y,\delta r))} \leq 2.
\end{equation}
We remark that the constant $r_{x,\lambda}$ depends on both $x$ and $\lambda$
as suggested by the notation.

We set $h_{B(y,s)} := m(B(y,s))^{-1} \int_{B(y,s)} h\,dm$ for $(y,s)\in X\times(0,\infty)$.
Let $r \in (0,r_{x,\lambda})$ and let $N$ be an $r/\lambda$-net in $(B(x,r),d)$.
Then for all $y_{1}, y_{2} \in N$ such that $d(y_{1},y_{2}) \leq 3 r/\lambda$,
\begin{align} \label{e:hm1}
&\bigl|h_{B(y_{1},r/\lambda)}-h_{B(y_{2},r/\lambda)}\bigr|^{2} \\
&=\biggl|\frac{1}{m(B(y_{1},r/\lambda))m(B(y_{2},r/\lambda))} \int_{B(y_{1},r/\lambda)} \int_{B(y_{2},r/\lambda)}(h(z_{1})-h(z_{2}))\,dm(z_{2})\,dm(z_{1})\biggr|^{2} \nonumber \\
&\leq \frac{1}{m(B(y_{1},r/\lambda))m(B(y_{2},r/\lambda))} \int_{B(y_{1},r/\lambda)} \int_{B(y_{2},r/\lambda)}|h(z_{1})-h(z_{2})|^{2}\,dm(z_{2})\,dm(z_{1}) \nonumber \\[-6pt]
	&\mspace{387.65mu} \textrm{(by the Cauchy--Schwarz inequality)} \nonumber \\
&\lesssim \frac{1}{m(B(y_{1},4r/\lambda))^{2}} \int_{B(y_{1},4r/\lambda)} \int_{B(y_{1},4r/\lambda)}|h(z_{1})-h(z_{2})|^{2}\,dm(z_{2})\,dm(z_{1}) \qquad \textrm{(by \hyperlink{vd}{$\on{VD}$})} \nonumber\\
&\lesssim \frac{ \Psi(r/\lambda)}{m(B(y_{1},4r/\lambda))} \int_{B(y_{1},4Ar/\lambda)} d\Gamma(h,h) \qquad \textrm{(by \hyperlink{pi}{$\on{PI}(\Psi)$} and Assumption \ref{a:reg})}\nonumber \\
&\leq C_{1} \Psi(r/\lambda) \qquad \textrm{(by \eqref{e:hm0} and \hyperlink{vd}{$\on{VD}$})}, \nonumber
\end{align}
where $C_{1}>0$ depends only on the constants in Assumption \ref{a:reg},
\hyperlink{vd}{$\on{VD}$} and \hyperlink{pi}{$\on{PI}(\Psi)$}.

Let $y_{1}, y_{2} \in N$ be arbitrary. Since $(X,d)$ is geodesic, approximating the
concatenation of a geodesic from $y_1$ to $x$ and a geodesic from $x$ to $y_2$ by using points
in $N$ as done in \cite[Proof of Lemma 2.5]{Kan}, we can choose $k\in\mathbb{N}$
and $\{z_{i}\}_{i=0}^{k} \subset N$ so that $k \leq 3 \lambda$, $z_{0} = y_{1}$,
$z_{k} = y_{2}$ and $d(z_{i}, z_{i+1}) \leq 3r/ \lambda$ for all $i \in \{0,\ldots,k-1\}$.
Therefore by the triangle inequality and \eqref{e:hm1}, we obtain
\begin{equation} \label{e:hm2}
\bigl| h_{B(y_{1},r/\lambda)} - h_{B(y_{2},r/\lambda)} \bigr|
	\leq \sum_{i=0}^{k-1} \bigl| h_{B(z_i,r/\lambda)} - h_{B(z_{i+1},r/\lambda)} \bigr|
	\leq 3C_{1}^{1/2} \lambda \sqrt{\Psi(r/\lambda)}.
\end{equation}

Let $y_{1} \in N$ be fixed. Combining \eqref{e:hm2} and \eqref{e:hm0} with
\hyperlink{vd}{$\on{VD}$} and \hyperlink{pi}{$\on{PI}(\Psi)$}, we conclude
\begin{align} \label{e:hm3}
&\int_{B(x,r)} |h-h_{B(x,r)}|^{2} \, dm \\
&\leq \int_{B(x,r)} |h-h_{B(y_{1},r/\lambda)}|^{2} \, dm
	\leq \sum_{y_{2} \in N} \int_{B(y_{2},r/\lambda)} |h-h_{B(y_1,r/\lambda)}|^{2} \, dm \nonumber\\ 
&\leq 2 \sum_{y_{2} \in N} \int_{B(y_{2},r/\lambda)} \Bigl( \bigl|h_{B(y_{1},r/\lambda)}-h_{B(y_{2},r/\lambda)}\bigr|^{2} + |h-h_{B(y_{2},r/\lambda)}|^{2} \Bigr) \,dm \nonumber\\
&\leq 2 \sum_{y_{2} \in N} \int_{B(y_{2},r/\lambda)} \Bigl( 9C_{1} \lambda^{2} \Psi(r/\lambda) + |h-h_{B(y_{2},r/\lambda)}|^{2} \Bigr) \, dm \qquad \textrm{(by \eqref{e:hm2})}\nonumber\\
&\lesssim \lambda^{2} \Psi(r/\lambda) m(B(x,r)) + \sum_{y_{2} \in N} \Psi(r/\lambda) \Gamma(h,h)(B(y_{2},Ar/\lambda))\qquad \textrm{(by \hyperlink{vd}{$\on{VD}$} and \hyperlink{pi}{$\on{PI}(\Psi)$})}\nonumber\\
&\lesssim \lambda^{2} \Psi(r/\lambda) m(B(x,r)) + \sum_{y_{2} \in N} \Psi(r/\lambda) m(B(y_{2},r/\lambda)) \qquad \textrm{(by \eqref{e:hm0} and \hyperlink{vd}{$\on{VD}$})} \nonumber\\
&\leq C_{2} \lambda^{2} \Psi(r/\lambda) m(B(x,r)) \qquad \textrm{(by \hyperlink{vd}{$\on{VD}$})}, \nonumber
\end{align}
where $C_{2}>0$ depends only on the constants in Assumption \ref{a:reg},
\hyperlink{vd}{$\on{VD}$} and \hyperlink{pi}{$\on{PI}(\Psi)$}.

On the other hand, by Lemma \ref{l:reverse-pi}, \eqref{e:reg}, \eqref{e:hm0}
and \hyperlink{vd}{$\on{VD}$}, for all $r \in (0,r_{x,\lambda})$ we have
\begin{equation} \label{e:hm4}
\int_{B(x,r)} |h-h_{B(x,r)}|^{2} \, dm
	\geq C_{3}^{-1} \Psi(r) \Gamma(h,h)(B(x,r/2)) \geq C_{4}^{-1} \Psi(r) m(B(x,r)),
\end{equation}
where $C_{3},C_{4}>0$ depend only on the constants in
Assumption \ref{a:reg}, \hyperlink{vd}{$\on{VD}$} and \hyperlink{cs}{$\on{CS}(\Psi)$}.
Now it follows from \eqref{e:hm3} and \eqref{e:hm4} that
\begin{equation*}
\frac{\lambda^{2}\Psi(r/\lambda)}{\Psi(r)}\geq C_{2}^{-1}C_{4}^{-1}
	\qquad \textrm{for all $\lambda > 4A$ and all $r \in (0,r_{x,\lambda})$,}
\end{equation*}
and hence $\liminf_{\lambda\to\infty}\liminf_{r\downarrow 0} \lambda^{2}\Psi(r/\lambda)/\Psi(r)\geq C_{2}^{-1}C_{4}^{-1}>0$,
which contradicts \eqref{e:case-sing} and completes the proof.
\end{proof}

The absolute continuity and singularity of energy measures are preserved under
linear combinations and norm convergence in $(\mathcal{F},\mathcal{E}_{1})$,
as stated in the following two lemmas.

\begin{lemma} \label{l:lin}
Let $(X,d,m,\mathcal{E},\mathcal{F})$ be a MMD space and let $\nu$ be a $\sigma$-finite
Borel measure on $X$. Let $f,g\in\mathcal{F}$ and $a,b\in\mathbb{R}$.
\begin{itemize}[label=\textup{(b)},align=left,leftmargin=*]
\item[\textup{(a)}]If $\Gamma(f,f) \ll \nu$ and $\Gamma(g,g) \ll \nu$, then $\Gamma(af+bg,af+bg) \ll \nu$.
\item[\textup{(b)}]If $\Gamma(f,f)\perp\nu$ and $\Gamma(g,g)\perp\nu$, then $\Gamma(af+bg,af+bg)\perp\nu$.
\end{itemize}
\end{lemma}

\begin{proof}
\begin{itemize}[label=\textup{(b)},align=left,leftmargin=*]
\item[\textup{(a)}]This is immediate from
	\eqref{e:bilinear-energy} and \eqref{e:Cauchy-Schwarz-energy}.
\item[\textup{(b)}]By $\Gamma(f,f)\perp\nu$ and $\Gamma(g,g)\perp\nu$ there exist
	Borel subsets $B_{1},B_{2}$ of $X$ such that $\Gamma(f,f)(B_{1})=\Gamma(g,g)(B_{2})=0$
	and $\nu(X\setminus B_{1})=\nu(X\setminus B_{2})=0$. Then $B:=B_{1}\cap B_{2}$
	satisfies $\Gamma(f,f)(B)=\Gamma(g,g)(B)=0$, hence $\Gamma(af+bg,af+bg)(B)=0$
	by \eqref{e:bilinear-energy} and \eqref{e:Cauchy-Schwarz-energy},
	and also $\nu(X\setminus B)=0$, proving $\Gamma(af+bg,af+bg)\perp\nu$.
\qedhere\end{itemize}
\end{proof}

\begin{lemma} \label{l:cont-ac-sing}
Let $(X,d,m,\mathcal{E},\mathcal{F})$ be a MMD space and let $\nu$ be a $\sigma$-finite
Borel measure on $X$. Let $\{f_{n}\}_{n=1}^{\infty}\subset\mathcal{F}$ and
$f\in\mathcal{F}$ satisfy $\lim_{n\to\infty}\mathcal{E}(f-f_{n},f-f_{n})=0$.
\begin{itemize}[label=\textup{(b)},align=left,leftmargin=*]
\item[\textup{(a)}]If $\Gamma(f_{n},f_{n}) \ll \nu$ for every $n \in \mathbb{N}$, then $\Gamma(f,f) \ll \nu$.
\item[\textup{(b)}]If $\Gamma(f_{n},f_{n}) \perp \nu$ for every $n \in \mathbb{N}$, then $\Gamma(f,f) \perp \nu$.
\end{itemize}
\end{lemma}

\begin{proof}
\begin{itemize}[label=\textup{(b)},align=left,leftmargin=*]
\item[\textup{(a)}]This is immediate from \cite[Proof of Lemma 2.2]{Hin10}.
\item[\textup{(b)}]For each $n\in\mathbb{N}$, by $\Gamma(f_{n},f_{n}) \perp \nu$ there exists
	a Borel subset $B_{n}$ of $X$ such that $\Gamma(f_{n},f_{n})(B_{n})=0$ and
	$\nu(X\setminus B_{n})=0$. Then $B := \bigcap_{n=1}^{\infty} B_{n}$ satisfies
	$\Gamma(f_{n},f_{n})(B)=0$ for all $n \in \mathbb{N}$ and $\nu(X\setminus B)=0$.
	By \eqref{e:Cauchy-Schwarz-energy}, \eqref{e:bilinear-energy} with $a=-b=1$
	and \eqref{e:EnergyMeasTotalMass},
	\begin{equation} \label{e:EnergyMeasConv}
	\begin{split}
	\Gamma(f,f)(B) &= \bigl|\Gamma(f,f)(B)^{1/2} - \Gamma(f_{n},f_{n})(B)^{1/2}\bigr|^{2}\\
	&\leq \Gamma(f-f_{n},f-f_{n})(B) \leq \mathcal{E}(f-f_{n},f-f_{n}) \xrightarrow{n\to\infty} 0,
	\end{split}
	\end{equation}
	so that $B$ satisfies both $\Gamma(f,f)(B)=0$ and $\nu(X\setminus B)=0$,
	proving $\Gamma(f,f) \perp \nu$.
\qedhere\end{itemize}
\end{proof}

We next show that any non-negative function in $\mathcal{F} \cap \mathcal{C}_{\mathrm{c}}(X)$
can be approximated in norm in $(\mathcal{F},\mathcal{E}_{1})$ by
``piecewise $\mathcal{E}$-harmonic functions'' whose energy measures charge
only their domains of $\mathcal{E}$-harmonicity. This approximation is used
together with Lemma \ref{l:cont-ac-sing}-(b) to extend the singularity of the energy
measures to all $f\in\mathcal{F}$ in Proposition \ref{p:approx} below, and is obtained
on the basis of the following fact from the theory of regular symmetric Dirichlet forms.

\begin{lemma} \label{l:harmonic-ext}
Let $(X,d,m,\mathcal{E},\mathcal{F})$ be a MMD space, let $U$ be an open subset of $X$
with $m(U)<\infty$ and let $F$ be a closed subset of $X$ with $F\subset U$. Then there
exists a linear map $H^{U}_{F}:\mathcal{F}_{U}\cap L^{\infty}(X,m)\to\mathcal{F}_{U}$
such that for any $f\in\mathcal{F}_{U}\cap L^{\infty}(X,m)$ with $f\geq 0$,
$H^{U}_{F}(f)=f$ $\mathcal{E}$-q.e.\ on $F$,
$H^{U}_{F}(f)$ is $\mathcal{E}$-harmonic on $U\setminus F$ and
$0\leq H^{U}_{F}(f)\leq\|f\|_{L^{\infty}(X,m)}$ $\mathcal{E}$-q.e.
\end{lemma}

\begin{proof}
Let $H^{U}_{F}$ be the map $H_{B}$ defined in \cite[Theorem 4.3.2]{FOT} with
$B:=F\cup(X\setminus U)$. It is a linear map from the \emph{extended Dirichlet space}
$\mathcal{F}_{e}$ to itself by \cite[Theorem 4.6.5]{FOT}, and for any $f\in\mathcal{F}_{e}$
with $f\geq 0$ we have $0\leq H^{U}_{F}(f)\leq\|f\|_{L^{\infty}(X,m)}$ $\mathcal{E}$-q.e.\ by
\cite[Lemma 2.1.4, Theorem 4.2.1-(ii) and Theorem 4.1.1]{FOT} and
$H^{U}_{F}(f)=f$ $\mathcal{E}$-q.e.\ on $B$ by
\cite[Theorem A.2.6-(i), Theorem 4.1.3 and Theorem 4.2.1-(ii)]{FOT}.
In particular, for any $f\in\mathcal{F}_{U}\cap L^{\infty}(X,m)$,
$H^{U}_{F}(f)=H^{U}_{F}(f^{+})-H^{U}_{F}(f^{-})\in L^{\infty}(X,m)$,
$H^{U}_{F}(f)=f$ $\mathcal{E}$-q.e.\ on $F$,
$H^{U}_{F}(f)=f=0$ $\mathcal{E}$-q.e.\ on $X\setminus U$, hence
$H^{U}_{F}(f)\in \mathcal{F}_{e}\cap L^{2}(X,m)=\mathcal{F}$
by $m(U)<\infty$ and \cite[Theorem 1.5.2-(iii)]{FOT},
thus $H^{U}_{F}(f)\in\mathcal{F}_{U}$, and $H^{U}_{F}(f)$ is $\mathcal{E}$-harmonic
on $X\setminus B=U\setminus F$ for any $f\in\mathcal{F}_{U}\cap L^{\infty}(X,m)$
by \cite[Theorem 4.6.5]{FOT}, completing the proof.
\end{proof}

\begin{proposition} \label{p:approx-harmonic}
Let $(X,d,m,\mathcal{E},\mathcal{F})$ be a MMD space.
Let $f\in\mathcal{F}\cap\mathcal{C}_{\mathrm{c}}(X)$ satisfy $f \geq 0$, and
for each $n\in\mathbb{N}$ set $F_{n}:=f^{-1}(2^{-n}\mathbb{Z})$ and define
$f_{n}\in\mathcal{F}_{X\setminus f^{-1}(0)}\cap L^{\infty}(X,m)$ by
\begin{equation}\label{e:approx-harmonic}
f_{n}=\sum_{k\in\mathbb{Z}\cap[0,2^{n}\|f\|_{\sup}]}f_{n,k},\mspace{11.38mu}
	\textrm{where}\mspace{11.38mu}f_{n,k}:=H^{f^{-1}((k2^{-n},\infty))}_{f^{-1}([(k+1)2^{-n},\infty))}\bigl((f-k2^{-n})^{+}\wedge 2^{-n}\bigr).
\end{equation}
Then for any $n\in\mathbb{N}$, $f_{n}=f$ $\mathcal{E}$-q.e.\ on $F_{n}$,
$f_{n}$ is $\mathcal{E}$-harmonic on $X\setminus F_{n}$, $\Gamma(f_{n},f_{n})(F_{n})=0$
and $|f-f_{n}|\leq 2^{-n}\one_{X\setminus f^{-1}(0)}$ $\mathcal{E}$-q.e.
Moreover, $\lim_{n\to\infty}\mathcal{E}_{1}(f-f_{n},f-f_{n})=0$.
\end{proposition}

\begin{proof}
Let $n\in\mathbb{N}$ and $k\in\mathbb{Z}\cap[0,2^{n}\|f\|_{\sup}]$. Since
$(f-k2^{-n})^{+}\wedge 2^{-n}\mspace{-1mu}\in\mspace{-1mu}\mathcal{F}_{f^{-1}((k2^{-n},\infty))}\mspace{-1mu}\cap\mathcal{C}_{\mathrm{c}}(X)$
by \cite[Theorem 1.4.1]{FOT}, we immediately see from Lemma \ref{l:harmonic-ext}
that $f_{n,k}$ is a well-defined element of $\mathcal{F}_{f^{-1}((k2^{-n},\infty))}$,
is $\mathcal{E}$-harmonic on $f^{-1}((k2^{-n},(k+1)2^{-n}))$ and satisfies
\begin{equation}\label{e:fnk-constant}
0\leq f_{n,k}\leq 2^{-n}\textrm{ $\mathcal{E}$-q.e.}\quad\textrm{and}\quad
f_{n,k}=
	\begin{cases}
		0 & \textrm{$\mathcal{E}$-q.e.\ on $f^{-1}([0,k2^{-n}])$,}\\
		2^{-n} & \textrm{$\mathcal{E}$-q.e.\ on $f^{-1}([(k+1)2^{-n},\infty))$.}
	\end{cases}
\end{equation}
In particular, $f_{n,k}$ is $\mathcal{E}$-harmonic on $X\setminus f^{-1}(\{k2^{-n},(k+1)2^{-n}\})$
by the strong locality of $(\mathcal{E},\mathcal{F})$ and
the fact that $g\one_{U}\in\mathcal{F}\cap\mathcal{C}_{\mathrm{c}}(X)$ and
$\supp_{m}[g\one_{U}]\subset U$ for any $g\in\mathcal{F}\cap\mathcal{C}_{\mathrm{c}}(X)$
with $\supp_{m}[g]\subset X\setminus f^{-1}(\{k2^{-n},(k+1)2^{-n}\})$ by
\cite[Exercise 1.4.1 and Theorem 1.4.2-(ii)]{FOT}, where $U$ denotes any one of
$f^{-1}([0,k2^{-n}))$, $f^{-1}((k2^{-n},(k+1)2^{-n}))$ and $f^{-1}(((k+1)2^{-n},\infty))$.
Thus $f_{n}\in\mathcal{F}_{X\setminus f^{-1}(0)}\cap L^{\infty}(X,m)$, $f_{n}$ is
$\mathcal{E}$-harmonic on $X\setminus F_{n}$, and it easily follows from \eqref{e:fnk-constant}
that $|f-f_{n}|\leq 2^{-n}\one_{X\setminus f^{-1}(0)}$ $\mathcal{E}$-q.e.\ and that
$f_{n}=f\in 2^{-n}\mathbb{Z}$ $\mathcal{E}$-q.e.\ on $F_{n}$, whence
$\Gamma(f_{n},f_{n})(F_{n})\leq\Gamma(f_{n},f_{n})(f_{n}^{-1}(2^{-n}\mathbb{Z}))=0$
by the absolute continuity of $\Gamma(f_{n},f_{n})(f_{n}^{-1}(\cdot))$ with respect to
the Lebesgue measure on $\mathbb{R}$ deduced from the strong locality of
$(\mathcal{E},\mathcal{F})$ and \cite[Theorem 4.3.8]{CF}.
Also, integrating the inequality $|f-f_{n}|^{2}\leq 4^{-n}\one_{X\setminus f^{-1}(0)}$ yields
$\|f-f_{n}\|_{L^{2}(X,m)}\leq 2^{-n}m(X\setminus f^{-1}(0))^{1/2}\xrightarrow{n\to\infty}0$.

Finally, for any $n,k\in\mathbb{N}$ with $n\leq k$, we have
$\mathcal{E}(f,f_{n})=\mathcal{E}(f_{n},f_{n})=\mathcal{E}(f_{k},f_{n})$
by the $\mathcal{E}$-harmonicity of $f_{n}$ on $X\setminus F_{n}$,
$f=f_{n}=f_{k}$ $\mathcal{E}$-q.e.\ on $F_{n}$ and \eqref{e:harmonic}, and therefore
\begin{gather}\label{e:approx-harmonic-bounded}
\mathcal{E}(f,f)=\mathcal{E}(f_{n},f_{n})+\mathcal{E}(f-f_{n},f-f_{n})\geq\mathcal{E}(f_{n},f_{n}),\\
\mathcal{E}(f_{k},f_{k})-\mathcal{E}(f_{n},f_{n})=\mathcal{E}(f_{k}-f_{n},f_{k}-f_{n})\geq 0.
\label{e:approx-harmonic-nondec}
\end{gather}
Then $\{\mathcal{E}(f_{n},f_{n})\}_{n=1}^{\infty}\subset[0,\mathcal{E}(f,f)]$ by
\eqref{e:approx-harmonic-bounded}, it is non-decreasing by \eqref{e:approx-harmonic-nondec}
and hence converges in $\mathbb{R}$, which together with \eqref{e:approx-harmonic-nondec}
and $\lim_{n\to\infty}\|f-f_{n}\|_{L^{2}(X,m)}=0$ implies that $\{f_{n}\}_{n=1}^{\infty}$
is a Cauchy sequence in the Hilbert space $(\mathcal{F},\mathcal{E}_{1})$. So
$\lim_{n\to\infty}\mathcal{E}_{1}(g-f_{n},g-f_{n})=0$ for some $g\in\mathcal{F}$,
which has to coincide with $f$ by $\lim_{n\to\infty}\|f-f_{n}\|_{L^{2}(X,m)}=0$.
\end{proof}

As mentioned above, we now prove the following proposition as the last main step.

\begin{proposition} \label{p:approx}
Let $(X,d,m,\mathcal{E},\mathcal{F})$ be a MMD space, and assume that
$\Gamma(h,h)|_{U} \perp m|_{U}$ for any open subset $U$ of $X$ and any
$h\in\mathcal{F}\cap L^{\infty}(X,m)$ that is $\mathcal{E}$-harmonic on $U$.
Then $\Gamma(f,f) \perp m$ for all $f \in \mathcal{F}$.
\end{proposition}

\begin{proof}
Since $\mathcal{F}\cap\mathcal{C}_{\mathrm{c}}(X)$ is norm dense in
$(\mathcal{F},\mathcal{E}_{1})$ by the regularity of $(\mathcal{E},\mathcal{F})$,
in view of Lemma \ref{l:cont-ac-sing}-(b) it suffices to consider the case of
$f\in\mathcal{F}\cap\mathcal{C}_{\mathrm{c}}(X)$. Also, writing
$f=f^{+}-f^{-}$ and noting that $f^{+},f^{-}\in\mathcal{F}\cap\mathcal{C}_{\mathrm{c}}(X)$
by \cite[Theorem 1.4.2-(i)]{FOT}, thanks to Lemma \ref{l:lin}-(b) we may assume without
loss of generality that $f\geq 0$.

Then for each $n\in\mathbb{N}$, setting $F_{n}:=f^{-1}(2^{-n}\mathbb{Z})$ and
defining $f_{n}\in\mathcal{F}_{X\setminus f^{-1}(0)}\cap L^{\infty}(X,m)$
by \eqref{e:approx-harmonic}, we have $\Gamma(f_{n},f_{n})(F_{n})=0$ and
the $\mathcal{E}$-harmonicity of $f_{n}$ on $X\setminus F_{n}$ by
Proposition \ref{p:approx-harmonic}, and therefore the assumption yields
$\Gamma(f_{n},f_{n})|_{X\setminus F_{n}}\perp m|_{X\setminus F_{n}}$, which
together with $\Gamma(f_{n},f_{n})(F_{n})=0$ implies $\Gamma(f_{n},f_{n}) \perp m$.
Now $\Gamma(f,f) \perp m$ follows by this fact, the norm convergence
$\lim_{n\to\infty}\mathcal{E}_{1}(f-f_{n},f-f_{n})=0$ from
Proposition \ref{p:approx-harmonic} and Lemma \ref{l:cont-ac-sing}-(b).
\end{proof}

\begin{proof}[Proof of Theorem \textup{\ref{t:main}-(a)}]
It is easy to verify that \hyperlink{vd}{$\on{VD}$} is preserved under a bi-Lipschitz change of
the metric and that so is \hyperlink{pi}{$\operatorname{PI}(\Psi)$} provided $\Psi$ satisfies
Assumption \ref{a:reg}. The same holds also for \hyperlink{cs}{$\operatorname{CS}(\Psi)$}
under Assumption \ref{a:reg} for $\Psi$ and \hyperlink{vd}{$\on{VD}$} by \cite[Lemma 5.7]{AB};
to be precise, here we need to use a slight variant of \cite[Lemma 5.7]{AB}
with the radius $r/2$ in its assumption replaced by $r/(2C^{2})$
for the constant $C\geq 1$ in the bi-Lipschitz equivalence of the metrics,
but \cite[Proof of Lemma 5.7]{AB} works also for this variant. Therefore using
Proposition \ref{p:cc}, we may assume without loss of generality that $d$ is geodesic,
and now it follows from Propositions \ref{p:harmonic} and \ref{p:approx} that
$\Gamma(f,f) \perp m$ for all $f \in \mathcal{F}$. In particular, for any
$f \in \mathcal{F}_{\on{loc}} \cap \mathcal{C}(X)$ with $\Gamma(f,f) \leq m$, we have
$\Gamma(f,f)(X)=0$, which together with \hyperlink{pi}{$\operatorname{PI}(\Psi)$}
and the relative compactness of $B(x,r)$ in $X$ for all $(x,r)\in X\times(0,\infty)$
implies that $f=a\one_{X}$ for some $a\in\mathbb{R}$.
Thus $d_{\on{int}}(x,y)=0$ for any $x,y\in X$ by \eqref{e:dint}.
\end{proof}

The above proof of Theorem \ref{t:main}-(a) easily extends to the more general situation
where the Poincar\'{e} inequality \hyperlink{pi}{$\operatorname{PI}(\Psi_{\on{PI}})$}
and the cutoff Sobolev inequality \hyperlink{cs}{$\operatorname{CS}(\Psi_{\on{CS}})$}
are assumed to hold with respect to possibly different space-time scale functions
$\Psi_{\on{PI}}$ and $\Psi_{\on{CS}}$, as follows.

\begin{theorem} \label{t:sing-two-scales}
Let $\Psi_{\on{PI}},\Psi_{\on{CS}}:[0,\infty)\to[0,\infty)$ be homeomorphisms satisfying
Assumption \textup{\ref{a:reg}} and let $(X,d,m,\mathcal{E},\mathcal{F})$ be a MMD space
satisfying \hyperlink{vd}{$\on{VD}$}, \hyperlink{pi}{$\operatorname{PI}(\Psi_{\on{PI}})$}
and \hyperlink{cs}{$\operatorname{CS}(\Psi_{\on{CS}})$}.
Assume further that $(X,d)$ satisfies the chain condition and that
\begin{equation} \label{e:case-sing-two-scales}
\liminf_{\lambda \to \infty} \liminf_{r \downarrow 0}
	\frac{\lambda^{2} \Psi_{\on{PI}}(r/\lambda)}{\Psi_{\on{CS}}(r)}=0.
\end{equation}
Then $\Gamma(f,f) \perp m$ for all $f \in \mathcal{F}$.
\end{theorem}

\begin{proof}
It is straightforward to see that the proof of Proposition \ref{p:harmonic} extends to
the present situation under the additional assumption that $d$ is geodesic. The rest of
the proof goes in exactly the same way as the above proof of Theorem \ref{t:main}-(a).
\end{proof}

\section{Absolute continuity} \label{s:ac}
In this section, we give the proof of Theorem \ref{t:main}-(b), namely the
``mutual absolute continuity'' between the symmetric measure $m$ and
the energy measures under the assumption \eqref{e:case-ac}.
\emph{In this section we do NOT assume that $(X,d)$ satisfies the chain condition except
in Proposition \textup{\ref{p:bilip}}.} Recall that we always have $\diam(X,d)\in(0,\infty]$
for a metric measure space $(X,d,m)$ by our standing assumption that $\#X\geq 2$.

We begin with the following lemma, which shows that the estimate \eqref{e:case-ac}
can be upgraded to the Gaussian space-time scaling \eqref{e:ge2} at small scales.

\begin{lemma} \label{l:ge2}
Let $(X,d,m,\mathcal{E},\mathcal{F})$ be a MMD space satisfying \hyperlink{vd}{$\on{VD}$},
\hyperlink{pi}{$\on{PI}(\Psi)$} and \hyperlink{cs}{$\on{CS}(\Psi)$}, and assume further
that $\Psi$ satisfies \eqref{e:case-ac}. Then there exist $r_{1} \in (0,\diam(X,d))$
and $C_{1} \geq 1$ such that \eqref{e:ge2} holds.
\end{lemma}

\begin{proof}
By \cite[Corollary 1.10]{Mur}, there exists $C_{1}\geq 1$ such that
\begin{equation} \label{e:ge1}
C_{1}^{-1} \frac{r^{2}}{s^{2}} \leq \frac{\Psi(r)}{\Psi(s)} \qquad
	\textrm{for all $0<s \leq r < \diam(X,d)$.}
\end{equation}
The desired upper bound on $\Psi(r)$ follows immediately from \eqref{e:ge1}.
The lower bound on $\Psi(r)$ for $r \in (0,\diam(X,d))$ follows by letting
$s \downarrow 0$ in \eqref{e:ge1} and using \eqref{e:case-ac} to obtain
\begin{equation*}
\frac{\Psi(r)}{r^{2}} \geq C_{1}^{-1}\limsup_{s \downarrow 0} \frac{\Psi(s)}{s^{2}}>0,
\end{equation*}
completing the proof.
\end{proof}

The upper inequality in \eqref{e:bi-Lipschitz-d-dint} is obtained from
\hyperlink{vd}{$\on{VD}$}, \hyperlink{pi}{$\operatorname{PI}(\Psi)$} and \eqref{e:ge2}, as follows.

\begin{lemma} \label{l:pi-dist}
Let $(X,d,m,\mathcal{E},\mathcal{F})$ be a MMD space satisfying \hyperlink{vd}{$\on{VD}$}
and \hyperlink{pi}{$\operatorname{PI}(\Psi)$}, and assume further that $\Psi$ satisfies
\eqref{e:ge2}. Then there exist $C,r_{0}>0$ such that 
$d_{\on{int}}(x,y) \le C d(x,y)$ for all $x,y\in X$ with $d(x,y) < r_{0}$.
\end{lemma}

\begin{proof}
Let $f \in \mathcal{F}_{\loc} \cap \mathcal{C}(X)$ satisfy $\Gamma(f,f) \leq m$. Then
by \cite[Lemma 2.4]{Mur} (see also \cite[Lemma 5.15]{HK98}), there exists $C>0$ such that
\begin{equation} \label{e:pi-dist}
|f(x)-f(y)| \leq C \sqrt{\Psi(r)} \qquad \textrm{for all $x,y \in X$ and $r>0$ with $d(x,y) \leq C^{-1}r$.}
\end{equation}
The desired estimate follows from \eqref{e:pi-dist}, \eqref{e:ge2} and \eqref{e:dint}.
\end{proof}

On the other hand, the lower inequality in \eqref{e:bi-Lipschitz-d-dint} follows from
\hyperlink{vd}{$\on{VD}$}, \hyperlink{cs}{$\operatorname{CS}(\Psi)$} and \eqref{e:ge2}
as stated in the following lemma, which also establishes standard properties of the
functions $(1-r^{-1}d(x,\cdot))^{+}$ in studying Gaussian heat kernel estimates
as a key step of the proof of the ``mutual absolute continuity''
between the symmetric measure $m$ and the energy measures.

\begin{lemma} \label{l:dist}
Let $(X,d,m,\mathcal{E},\mathcal{F})$ be a MMD space satisfying \hyperlink{vd}{$\on{VD}$}
and \hyperlink{cs}{$\operatorname{CS}(\Psi)$}, and assume further that $\Psi$ satisfies
\eqref{e:ge2}. Then there exist $C,r_{0}>0$ such that for all
$(x,r) \in X\times(0,r_{0})$, the function $f_{x,r}:= (1- r^{-1}d(x,\cdot))^{+}$
satisfies $f_{x,r} \in \mathcal{F}$ and $\Gamma(f_{x,r},f_{x,r}) \leq C^{2} r^{-2} m$.
In particular, $d_{\on{int}}(x,y) \geq C^{-1} d(x,y)$ for all $x,y \in X$ with $d(x,y)\wedge(Cd_{\on{int}}(x,y))<r_{0}$.
\end{lemma}

\begin{proof}
Let $r_{1}>0$ and $C_{1}\geq 1$ be as in \eqref{e:ge2}, $(x,r) \in X\times(0,r_{1})$
and $n\in\mathbb{N}\setminus\{1\}$. For each $i\in\{1,\ldots,n-1\}$, let
$\varphi_{i,n}\in\mathcal{F}$ be a cutoff function for $B(x, ir/n) \subset B(x,(i+1)r/n)$
as given in \hyperlink{cs}{$\on{CS}(\Psi)$} and set $U_{i,n}:= B(x,(i+1)r/n) \setminus B(x,ir/n)$,
so that by \hyperlink{cs}{$\on{CS}(\Psi)$} we have
\begin{equation} \label{e:si1}
\int_{X} g^{2}\, d\Gamma(\varphi_{i,n},\varphi_{i,n})
	\leq \frac{1}{8} \int_{U_{i,n}} d\Gamma(g,g) + \frac{C_{S}}{\Psi(r/n)} \int_{U_{i,n}} g^{2}\,dm
\end{equation}
for all $g \in \mathcal{F}$. Set
\begin{equation*}
\varphi_{n} := \frac{1}{n-1}\sum_{i=1}^{n-1} \varphi_{i,n},
\end{equation*}
so that $0 \leq \varphi_{n} \leq 1$ $m$-a.e., $\supp_{m}[\varphi_{n}] \subset B(x,r)$ and
\begin{equation} \label{e:si4}
|\varphi_{n}-f_{x,r}| \leq 2n^{-1}\one_{B(x,r)} \quad \textrm{$m$-a.e.}
\end{equation}
By the strong locality of $(\mathcal{E},\mathcal{F})$, \cite[Corollary 3.2.1]{FOT}
(or \cite[Theorem 4.3.8]{CF}) and \eqref{e:Cauchy-Schwarz-energy}, we have 
\begin{equation} \label{e:si2}
\Gamma(\varphi_{n},\varphi_{n}) = (n-1)^{-2}\sum_{i=1}^{n-1} \Gamma(\varphi_{i,n},\varphi_{i,n}).
\end{equation}
Combining \eqref{e:si1}, \eqref{e:si2} and \eqref{e:ge2}, we obtain
\begin{equation} \label{e:si3}
\begin{split}
\int_{X} g^{2}\, d\Gamma(\varphi_{n},\varphi_{n}) &\leq \frac{(n-1)^{-2}}{8} \int_{B(x,r)} d\Gamma(g,g)
	+ \frac{C_{S}(n-1)^{-2}}{\Psi(r/n)} \int_{B(x,r)} g^{2}\,dm \\
&\leq \frac{(n-1)^{-2}}{8} \int_{B(x,r)} d\Gamma(g,g) + \frac{4C_{1}C_{S}}{r^{2}} \int_{B(x,r)} g^{2}\,dm
\end{split}
\end{equation}
for all $g \in \mathcal{F}$. Therefore choosing $g\in\mathcal{F}\cap\mathcal{C}_{\mathrm{c}}(X)$
with $g=1$ on $B(x,r)$, which exists by the regularity of
$(\mathcal{E},\mathcal{F})$ and \cite[Exercise 1.4.1]{FOT}, and noting that
$\Gamma(\varphi_{n},\varphi_{n})(X\setminus B(x,r))=\Gamma(g,g)(B(x,r))=0$
by $\supp_{m}[\varphi_{n}] \subset B(x,r)$, the strong locality of $(\mathcal{E},\mathcal{F})$
and \cite[Corollary 3.2.1]{FOT} (or \cite[Theorem 4.3.8]{CF}), we see from \eqref{e:si3} that
\begin{equation*}
\mathcal{E}_{1}(\varphi_{n},\varphi_{n})
	\leq \Bigl( \frac{4 C_{1} C_{S}}{r^{2}} + 1 \Bigr) m(B(x,r))
	\qquad \textrm{for all $n \in \mathbb{N}\setminus\{1\}$.}
\end{equation*}
Hence by the Banach--Saks theorem \cite[Theorem A.4.1-(i)]{CF} there exists a
subsequence $\{\varphi_{n_{k}}\}_{k=1}^{\infty}$ of $\{\varphi_{n}\}_{n=2}^{\infty}$
such that its Ces\`{a}ro mean sequence
\begin{equation*}
\psi_{i} := \frac{1}{i} \sum_{k=1}^{i} \varphi_{n_{k}}, \quad i\in\mathbb{N},
\end{equation*}
converges in norm in $(\mathcal{F},\mathcal{E}_{1})$ as $i \to \infty$, but then its
limit must be $f_{x,r}$ by \eqref{e:si4} and in particular $f_{x,r} \in \mathcal{F}$.
On the other hand, by \eqref{e:bilinear-energy} and the Cauchy--Schwarz
inequality similar to \eqref{e:Cauchy-Schwarz-energy}, we have the triangle inequality
\begin{equation} \label{e:triangle-energy}
\biggl| \Bigl( \int_{X} g^{2}\,d\Gamma(f_{1},f_{1}) \Bigr)^{1/2} - \Bigl( \int_{X} g^{2}\,d\Gamma(f_{2},f_{2}) \Bigr)^{1/2} \biggr|
	\leq \Bigl( \int_{X} g^{2} \, d\Gamma(f_{1}-f_{2},f_{1}-f_{2}) \Bigr)^{1/2}
\end{equation}
for all $f_{1},f_{2} \in \mathcal{F}$ and all bounded Borel measurable function
$g:X\to\mathbb{R}$. Combining \eqref{e:triangle-energy} and \eqref{e:EnergyMeasTotalMass} 
with $\lim_{i \to \infty} \mathcal{E}_{1}( f_{x,r} - \psi_{i},  f_{x,r} - \psi_{i})=0$
in the same way as \eqref{e:EnergyMeasConv}, we obtain
\begin{align} \label{e:si5}
&\int_{X} g^{2} \,d\Gamma(f_{x,r},f_{x,r})
	=\lim_{i \to \infty}\int_{X} g^{2} \,d\Gamma(\psi_{i},\psi_{i}) \\
&\leq \liminf_{i \to \infty} \frac{1}{i} \sum_{k=1}^{i} \int_{X} g^{2} \, d\Gamma(\varphi_{n_k},\varphi_{n_k}) \quad \textrm{(by \eqref{e:triangle-energy} and the Cauchy--Schwarz inequality)} \nonumber\\
&\leq \lim_{i \to \infty} \frac{1}{i} \sum_{k=1}^{i} \biggl( \frac{(n_{k}-1)^{-2}}{8} \int_{B(x,r)} d\Gamma(g,g) + \frac{4 C_{1} C_{S}}{r^{2}} \int_{B(x,r)} g^{2} \,dm \biggr) \qquad \textrm{(by \eqref{e:si3})} \nonumber\\
&= \frac{4 C_{1} C_{S}}{r^{2}} \int_{B(x,r)} g^{2} \,dm \qquad \textrm{for all $g \in \mathcal{F} \cap \mathcal{C}_{\mathrm{c}}(X)$.} \nonumber
\end{align}
Since $\mathcal{F} \cap \mathcal{C}_{\mathrm{c}}(X)$ is dense in $(\mathcal{C}_{\mathrm{c}}(X),\|\cdot\|_{\sup})$
by the regularity of $(\mathcal{E},\mathcal{F})$, it follows from \eqref{e:si5} that
\begin{equation} \label{e:distance-energy}
\Gamma(f_{x,r},f_{x,r}) \leq 4 C_{1} C_{S} r^{-2} m.
\end{equation}

In particular, for all $(x,r)\in X\times(0,r_{1})$, the function
\begin{equation*}
\hat{f}_{x,r} := r (4C_{1} C_{S})^{-1/2} f_{x,r}
\end{equation*}
satisfies $\hat{f}_{x,r} \in \mathcal{F}\cap\mathcal{C}(X)$ and
$\Gamma( \hat{f}_{x,r}, \hat{f}_{x,r}) \leq m$ by \eqref{e:distance-energy},
and we therefore obtain
\begin{equation} \label{e:compdist1}
d_{\on{int}}(x,y) \geq \hat{f}_{x,r}(x) - \hat{f}_{x,r}(y) = (4C_{1} C_{S})^{-1/2} r
	\qquad \textrm{for all $y\in X$ with $d(x,y) \geq r$}
\end{equation}
in view of \eqref{e:dint}. Thus for each $x,y\in X$, if $d(x,y)\geq r_{1}$ then
$(4 C_{1} C_{S})^{1/2} d_{\on{int}}(x,y) \geq r_{1}$ by \eqref{e:compdist1},
hence if $(4 C_{1} C_{S})^{1/2} d_{\on{int}}(x,y) < r_{1}$ then $d(x,y) < r_{1}$, and
if in turn $d(x,y) < r_{1}$ then $d_{\on{int}}(x,y) \geq (4C_{1} C_{S})^{-1/2} d(x,y)$
either by using \eqref{e:compdist1} with $r=d(x,y)\in(0,r_{1})$ or by $d(x,y)=0$,
completing the proof.
\end{proof}

We also need the following lemma for the proof of the absolute continuity of the
energy measures achieved as Proposition \ref{p:dom} below. Recall the notion of
an $\varepsilon$-net in a metric space $(X,d)$ introduced in Definition \ref{d:net}.

\begin{lemma}[Lipschitz partition of unity] \label{l:lip}
Let $(X,d,m,\mathcal{E},\mathcal{F})$ be a MMD space satisfying \hyperlink{vd}{$\on{VD}$}
and \hyperlink{cs}{$\operatorname{CS}(\Psi)$}, and assume further that $\Psi$ satisfies
\eqref{e:ge2}. Then there exist $C,r_{0}>0$ such that for any $\varepsilon \in (0,r_{0})$
and any $\varepsilon$-net $N\subset X$ in $(X,d)$ there exists
$\{\varphi_{z}\}_{z \in N} \subset \mathcal{F} \cap \mathcal{C}_{\mathrm{c}}(X)$
with the following properties:
\begin{itemize}[label=\textup{(b)},align=left,leftmargin=*]
\item[\textup{(a)}]$\sum_{z \in N} \varphi_{z}(x) = 1$ for all $x\in X$.
\item[\textup{(b)}]$0 \leq \varphi_{z}(x) \leq \one_{B(z,2\varepsilon)}(x)$ for all $x\in X$ and all $z \in N$.
\item[\textup{(c)}]$\varphi_{z}$ is $C\varepsilon^{-1}$-Lipschitz for all $z \in N$, i.e.,
	$|\varphi_{z}(x)-\varphi_{z}(y)|\leq C\varepsilon^{-1}d(x,y)$ for all $x,y\in X$.
\item[\textup{(d)}]$\Gamma(\varphi_{z},\varphi_{z}) \leq C \varepsilon^{-2} m$ for all $z \in N$.
\item[\textup{(e)}]$\mathcal{E}(\varphi_{z},\varphi_{z}) \leq C \varepsilon^{-2} m(B(z,\varepsilon))$ for all $z \in N$.
\end{itemize}
\end{lemma}

\begin{proof}
Let $r_{0}>0$ be the constant from Lemma \ref{l:dist} and let
$f_{x,r}\in\mathcal{F}\cap\mathcal{C}(X)$ be as defined in Lemma \ref{l:dist}
for each $(x,r)\in X\times(0,r_{0})$. Let $\varepsilon \in (0,r_{0}/2)$ and
let $N\subset X$ be an $\varepsilon$-net in $(X,d)$. Noting that
\begin{equation} \label{e:lp1}
\frac{1}{2} \leq \sum_{w \in N} f_{w,2\varepsilon}(y)
	= \sum_{w \in N\cap B(z,4\varepsilon)} f_{w,2\varepsilon}(y)
	\leq \#(N\cap B(z,4\varepsilon)) \lesssim 1
\end{equation}
for all $z\in X$ and all $y \in B(z,2\varepsilon)$
by $\bigcup_{w \in N} B(w,\varepsilon)=X$ and \hyperlink{vd}{$\on{VD}$}, we define
\begin{equation} \label{e:lip-dfn}
\varphi_{z} := \frac{f_{z,2\varepsilon}}{\sum_{w \in N} f_{w,2\varepsilon}}
	=\frac{f_{z,2\varepsilon}}{\sum_{w \in N\cap B(z,4\varepsilon)} f_{w,2\varepsilon}}
	\qquad\textrm{for each $z\in N$,}
\end{equation}
so that properties (a) and (b) obviously hold and
$\{\varphi_{z}\}_{z \in N} \subset \mathcal{F}\cap\mathcal{C}_{\mathrm{c}}(X)$
by \cite[Exercise I.4.16 (or Corollary I.4.13)]{MR} and the relative compactness
of $B(z,2\varepsilon)$ in $X$. The estimate (d) follows easily from the chain rule
\cite[Theorem 3.2.2]{FOT} for $\Gamma$, the Cauchy--Schwarz inequality similar to
\eqref{e:Cauchy-Schwarz-energy}, \eqref{e:lp1} and Lemma \ref{l:dist}, and
the estimate (e) is an immediate consequence of \eqref{e:EnergyMeasTotalMass}, (b),
\cite[Corollary 3.2.1]{FOT} (or \cite[Theorem 4.3.8]{CF}), (d) and \hyperlink{vd}{$\on{VD}$}.

It remains to prove (c). First, note that by the triangle inequality,
$f_{z,2\varepsilon}$ is $(2\varepsilon)^{-1}$-Lipschitz for all $z \in X$, i.e.,
\begin{equation} \label{e:lp2}
|f_{z,2\varepsilon}(x)-f_{z,2\varepsilon}(y)| \leq (2\varepsilon)^{-1} d(x,y)
	\qquad \textrm{for all $x,y,z \in X$.}
\end{equation}
Let $z\in N$ and $x,y\in X$. If $d(x,y) \geq \varepsilon$, then 
\begin{equation} \label{e:lp3}
|\varphi_{z}(x)-\varphi_{z}(y)| \leq 1 \leq \varepsilon^{-1} d(x,y).
\end{equation}
On the other hand, if $d(x,y) < \varepsilon$, then
\begin{align} \label{e:lp4}
&|\varphi_{z}(x)-\varphi_{z}(y)| \\
&\leq \biggl|\frac{f_{z,2\varepsilon}(x)}{\sum_{w \in N} f_{w,2\varepsilon}(x)}-\frac{f_{z,2\varepsilon}(y)}{\sum_{w \in N} f_{w,2\varepsilon}(x)}\biggr|
	+ \biggl|\frac{f_{z,2\varepsilon}(y)}{\sum_{w \in N} f_{w,2\varepsilon}(x)}-\frac{f_{z,2\varepsilon}(y)}{\sum_{w \in N} f_{w,2\varepsilon}(y)}\biggr| \nonumber\\
&\leq \varepsilon^{-1} d(x,y) + \biggl|\frac{1}{\sum_{w \in N} f_{w,2\varepsilon}(x)}-\frac{1}{\sum_{w \in N} f_{w,2\varepsilon}(y)}\biggr| \qquad \textrm{(by \eqref{e:lp1} and \eqref{e:lp2})} \nonumber\\
&\leq \varepsilon^{-1} d(x,y) + 4 \Biggl|\sum_{w \in N \cap B(x,4\varepsilon)}\bigl(f_{w,2\varepsilon}(y) - f_{w,2\varepsilon}(x)\bigr) \Biggr| \quad \textrm{(by \eqref{e:lp1} and $d(x,y) < \varepsilon$)} \nonumber\\
& \lesssim \varepsilon^{-1} d(x,y) \qquad \textrm{(by \eqref{e:lp2} and \eqref{e:lp1}).} \nonumber
\end{align}
Combining \eqref{e:lp3} and \eqref{e:lp4}, we obtain (c).
\end{proof}

\begin{proposition}[Energy dominance of $m$] \label{p:dom}
Let $(X,d,m,\mathcal{E},\mathcal{F})$ be a MMD space satisfying \hyperlink{vd}{$\on{VD}$},
\hyperlink{pi}{$\on{PI}(\Psi)$} and \hyperlink{cs}{$\on{CS}(\Psi)$}, and assume further
that $\Psi$ satisfies \eqref{e:case-ac}. Then $m$ is an energy-dominant measure of
$(\mathcal{E},\mathcal{F})$, that is, $\Gamma(f,f) \ll m$ for all $f \in \mathcal{F}$.
\end{proposition}

\begin{proof}
Since $\mathcal{F} \cap \mathcal{C}_{\mathrm{c}}(X)$ is dense in
$(\mathcal{F},\mathcal{E}_{1})$ by the regularity of $(\mathcal{E},\mathcal{F})$,
by Lemma \ref{l:cont-ac-sing}-(a) it suffices to show that $\Gamma(f,f) \ll m$
for all $f \in \mathcal{F} \cap \mathcal{C}_{\mathrm{c}}(X)$.

Let $f \in \mathcal{F} \cap \mathcal{C}_{\mathrm{c}}(X)$. Noting that Lemma \ref{l:lip}
is applicable by Lemma \ref{l:ge2}, let $r_{1},r_{0}>0$ be the constants in
Lemmas \ref{l:ge2} and \ref{l:lip}, respectively. Let $n \in \mathbb{N}$ satisfy
$4 n^{-1} < r_{1}\wedge r_{0}$, let $N_{n} \subset X$ be an $n^{-1}$-net in $(X,d)$
and let $\{\varphi_{z}\}_{z \in N_{n}}$ be the Lipschitz partition of unity
as given in Lemma \ref{l:lip}. We define
\begin{equation} \label{e:dom1}
f_{n} := \sum_{z \in N_{n}} f_{B(z,n^{-1})} \varphi_{z}, \quad \textrm{ where}
	\quad f_{B(z,n^{-1})} := \frac{1}{m(B(z,n^{-1}))} \int_{B(z,n^{-1})} f\,dm,
\end{equation}
so that $f_{n}$ is in fact a finite linear combination of $\{\varphi_{z}\}_{z \in N_{n}}$
by the relative compactness of $\bigcup_{x\in\supp_{m}[f]}B(x,n^{-1})$ in $X$ and hence
satisfies $f_{n}\in\mathcal{F}\cap\mathcal{C}_{\mathrm{c}}(X)$ and, by Lemma \ref{l:lin}-(a),
\begin{equation} \label{e:dom2}
\Gamma(f_{n},f_{n}) \ll m.
\end{equation}
Since $\|f_{n}\|_{\sup}\leq\|f\|_{\sup}$ by Lemma \ref{l:lip}-(a),(b), we easily see that
\begin{equation} \label{e:fn-Lipschitz}
|f_{n}(x)-f_{n}(y)| \lesssim n\|f\|_{\sup} d(x,y) \qquad \textrm{for any $x,y\in X$}
\end{equation}
by treating the case of $d(x,y)\geq n^{-1}$ and that of $d(x,y)<n^{-1}$ separately as
in \eqref{e:lp3} and \eqref{e:lp4} and using Lemma \ref{l:lip}-(b),(c)
and \hyperlink{vd}{$\on{VD}$} for the latter case, and $f_{n}$ is thus Lipschitz.
Furthermore by Lemma \ref{l:lip}-(a),(b), for any $x\in X$ we have
\begin{align*}
|f_{n}(x)-f(x)|
	&= \Bigl| \sum\nolimits_{z \in N_{n} \cap B(x,2n^{-1})}\bigl(f_{B(z,n^{-1})}-f(x)\bigr)\varphi_{z}(x) \Bigr|\\
&\leq \sum\nolimits_{z \in N_{n} \cap B(x,2n^{-1})}\bigl|f_{B(z,n^{-1})}-f(x)\bigr|\varphi_{z}(x)\\
&\leq \sup \bigl\{ |f(w)-f(x)| \bigm| w \in B(x, 3n^{-1}) \bigr\},
\end{align*}
which together with the uniform continuity of $f\in\mathcal{C}_{\mathrm{c}}(X)$ on $X$ yields
\begin{equation} \label{e:dom3}
\|f_{n}-f\|_{\sup}
	\leq \sup \bigl\{ |f(z)-f(w)| \bigm|\textrm{$z,w\in X$, $d(z,w) < 3n^{-1}$}\bigr\}
	\xrightarrow{n\to\infty} 0.
\end{equation}
Also, choosing $(x_{0},r) \in X \times (0,\infty)$ so that $\supp_{m}[f] \subset B(x_{0},r)$,
we have $\supp_{m}[f_{n}] \subset B(x_{0},r+4)$ by Lemma \ref{l:lip}-(b), and therefore
from \eqref{e:dom3} we obtain
\begin{equation} \label{e:dom4}
\|f_{n}-f\|_{L^{2}(X,m)} \leq \|f_{n}-f\|_{\sup} m(B(x_{0},r+4))^{1/2} \xrightarrow{n\to\infty} 0.
\end{equation}

On the other hand, using \hyperlink{pi}{$\on{PI}(\Psi)$} together with
\hyperlink{vd}{$\on{VD}$} and Lemma \ref{l:ge2} in the same way as \eqref{e:hm1},
for all $z,w \in N_{n}$ with $d(z,w) \le 3 n^{-1}$ we have
\begin{equation} \label{e:dom5}
\bigl|f_{B(z,n^{-1})}- f_{B(w,n^{-1})}\bigr|^{2}
	\lesssim \frac{n^{-2}}{m(B(z,n^{-1}))} \int_{B(z,4 An^{-1})} d\Gamma(f,f),
\end{equation}
where $A\geq 1$ is the constant in \hyperlink{pi}{$\on{PI}(\Psi)$}.
For each $z\in N_{n}$, observing that
\begin{equation*}
f_{n}(x) = f_{B(z,n^{-1})} + \sum_{w\in N_{n} \cap B(z,3n^{-1})} \bigl(f_{B(w,n^{-1})} - f_{B(z,n^{-1})}\bigr)\varphi_{w}(x)
	\quad \textrm{ for all $x \in B(z,n^{-1})$}
\end{equation*}
by Lemma \ref{l:lip}-(a),(b), we see from the strong locality of $(\mathcal{E},\mathcal{F})$,
\cite[Corollary 3.2.1]{FOT} (or \cite[Theorem 4.3.8]{CF}), \eqref{e:triangle-energy}
and the Cauchy--Schwarz inequality that
\begin{align} \label{e:dom6}
&\Gamma(f_{n},f_{n})\bigl(B(z,n^{-1})\bigr) \\
&\leq \#\bigl(N_{n} \cap B(z,3n^{-1})\bigr) \sum_{w \in N_{n} \cap B(z,3n^{-1})} \bigl|f_{B(w,n^{-1})} - f_{B(z,n^{-1})}\bigr|^{2} \Gamma(\varphi_{w},\varphi_{w})\bigl(B(z,n^{-1})\bigr) \nonumber \\ 
&\lesssim \Gamma(f,f)\bigl(B(z,4An^{-1})\bigr) \qquad \textrm{(by \hyperlink{vd}{$\on{VD}$}, \eqref{e:dom5} and Lemma \ref{l:lip}-(d)).} \nonumber
\end{align}
Since $X=\bigcup_{z\in N_{n}}B(z,n^{-1})$ and $\sum_{z \in N_{n}} \one_{B(z,4 A n^{-1})} \lesssim 1$
by \hyperlink{vd}{$\on{VD}$}, from \eqref{e:EnergyMeasTotalMass} and \eqref{e:dom6} we obtain
\begin{equation} \label{e:dom7}
\mathcal{E}(f_{n},f_{n}) \leq \sum_{z \in N_{n}} \Gamma(f_{n},f_{n})\bigl(B(z,n^{-1})\bigr)
	\lesssim \sum_{z \in N_{n}} \Gamma(f,f)\bigl(B(z,4An^{-1})\bigr) \lesssim \mathcal{E}(f,f).
\end{equation}

It follows from \eqref{e:dom4} and \eqref{e:dom7} that $\{f_{n}\}_{n>4(r_{1}\wedge r_{0})^{-1}}$
is a bounded sequence in $(\mathcal{F},\mathcal{E}_{1})$, and hence by the Banach--Saks theorem
\cite[Theorem A.4.1-(i)]{CF} there exists a subsequence 
$\{f_{n_{k}}\}_{k=1}^{\infty}$ of $\{f_{n}\}_{n>4(r_{1}\wedge r_{0})^{-1}}$ such that
its Ces\`{a}ro mean sequence $\{i^{-1}\sum_{k=1}^{i}f_{n_{k}}\}_{i=1}^{\infty}$
converges in norm in $(\mathcal{F},\mathcal{E}_{1})$, but then the limit must
necessarily be $f$ by \eqref{e:dom4}. Now by \eqref{e:dom2}, Lemma \ref{l:lin}-(a) and
Lemma \ref{l:cont-ac-sing}-(a), we obtain $\Gamma(f,f) \ll m$, completing the proof.
\end{proof}

\begin{remark} \label{rmk:dom}
The above proof of Proposition \ref{p:dom} is inspired by \cite[Proof of Proposition 4.7]{KST}.
Note that it also shows that $\mathcal{F} \cap \on{Lip}_{\mathrm{c}}(X,d)$ is dense
in $(\mathcal{F},\mathcal{E}_{1})$ in the situation of Proposition \ref{p:dom}, where
$\on{Lip}_{\mathrm{c}}(X,d):=\{f\in\mathcal{C}_{\mathrm{c}}(X)\mid\textrm{$f$ is Lipschitz with respect to $d$}\}$.
We remark that Proposition \ref{p:dom} and this denseness were proved also in \cite[Lemma 2.11]{ABCRST}
with a very similar proof under the additional a priori assumptions that $d$ is
the intrinsic metric $d_{\on{int}}$ and that $\Psi$ is given by $\Psi(r)=r^{2}$.
\end{remark}

\begin{proposition}[Minimality of $m$] \label{p:min}
Let $(X,d,m,\mathcal{E},\mathcal{F})$ be a MMD space satisfying \hyperlink{vd}{$\on{VD}$},
\hyperlink{pi}{$\on{PI}(\Psi)$} and \hyperlink{cs}{$\on{CS}(\Psi)$}, and assume further
that $\Psi$ satisfies \eqref{e:case-ac}. If $\nu$ is a minimal energy-dominant measure
of $(\mathcal{E},\mathcal{F})$, then $m \ll \nu$.
\end{proposition}

\begin{proof}
Let	$m=m_{a}+m_{s}$ be the Lebesgue decomposition of $m$ with respect to $\nu$,
so that $m_{a} \ll \nu$ and $m_{s} \perp \nu$. We are to show that $m_{s}(X)=0$,
which will yield $m=m_{a} \ll \nu$.

Noting that Lemma \ref{l:dist} is applicable by Lemma \ref{l:ge2}, let $r_{1}\in(0,\diam(X,d))$
and $C,r_{0}>0$ be the constants in Lemmas \ref{l:ge2} and \ref{l:dist}, respectively.
Then by Lemma \ref{l:dist}, for all $(x,r) \in X\times(0,r_{0})$ we have
$f_{x,r} := (1-r^{-1}d(x,\cdot))^{+}\in\mathcal{F}$ and
$\Gamma(f_{x,r},f_{x,r}) \leq C^{2}r^{-2}m$, which together with
$\Gamma(f_{x,r},f_{x,r}) \ll \nu \perp m_{s}$ implies that
\begin{equation} \label{e:mi1}
\Gamma(f_{x,r},f_{x,r}) \leq C^{2}r^{-2}m_{a}.
\end{equation}
On the other hand, for each $(x,r) \in X\times (0,r_{1}/2)$, by $B(x,r)\not=X$
(recall that $r_{1}\in(0,\diam(X,d))$) and \cite[Proof of Corollary 2.3]{Mur} there exists
$y\in B(x,3r/4)\setminus B(x,r/2)$, and then there exists $\delta\in(0,1)$
determined solely by the constant $C_{D}$ in \hyperlink{vd}{$\on{VD}$} such that
\begin{equation} \label{e:mi2}
1-(f_{x,r})_{B(x,r)} \geq \frac{m(B(y,r/4))}{4m(B(x,r))} \geq \delta
	\mspace{10.98mu}\textrm{by $B(y,r/4)\subset B(x,r)\setminus B(x,r/4)$ and \hyperlink{vd}{$\on{VD}$},}
\end{equation}
where $(f_{x,r})_{B(x,r)}:=m(B(x,r))^{-1}\int_{B(x,r)}f_{x,r}\,dm$.
Thus for all $(x,r) \in X\times (0,r_{1}/2)$ we have
$f_{x,r}-(f_{x,r})_{B(x,r)}\geq\delta/2$ on $B(x,\delta r/2)$ by \eqref{e:mi2} and hence
\begin{equation} \label{e:mi3}
\begin{split}
m(B(x,Ar)) \lesssim m(B(x,\delta r/2))
	&\lesssim \int_{B(x,r)} \bigl|f_{x,r}-(f_{x,r})_{B(x,r)}\bigr|^{2}\,dm \qquad \textrm{(by \hyperlink{vd}{$\on{VD}$})} \\
&\lesssim \Psi(r) \Gamma(f_{x,r},f_{x,r})(B(x,Ar)) \qquad \textrm{(by \hyperlink{pi}{$\on{PI}(\Psi)$})} \\
&\lesssim m_{a}(B(x,Ar)) \qquad \textrm{(by Lemma \ref{l:ge2} and \eqref{e:mi1}),}
\end{split}
\end{equation}
where $A\geq 1$ is the constant in \hyperlink{pi}{$\on{PI}(\Psi)$}.

Now assume to the contrary that $m_{s}(X)>0$. Then by $m_{a} \ll \nu \perp m_{s}$
and the inner regularity of $m_{s}$ (see, e.g., \cite[Theorem 2.18]{Rud}),
there exists a compact subset $K$ of $X$ such that $m_{s}(K)>0$ and $m_{a}(K)=0$.
Let $\varepsilon \in (0,r_{1}/2)$, set $K_{\varepsilon}:=\bigcup_{x\in K}B(x,\varepsilon)$
and let $N_{\varepsilon}$ be a $2\varepsilon$-net in $(K,d)$,
so that $K_{\varepsilon}$ is relatively compact in $X$,
$K \subset \bigcup_{x\in N_{\varepsilon}} B(x,2\varepsilon)$
and $B(x,\varepsilon)\cap B(y,\varepsilon)=\emptyset$ for any $x,y\in N_{\varepsilon}$
with $x\not=y$. Using these properties, we obtain
\begin{align*}
0 < m(K) \leq \mspace{-1.5mu} \sum_{x\in N_{\varepsilon}} m(B(x,2\varepsilon))
	&\lesssim \mspace{-1.5mu} \sum_{x\in N_{\varepsilon}} m(B(x,\varepsilon))
	\lesssim \mspace{-1.5mu} \sum_{x\in N_{\varepsilon}} m_{a}(B(x,\varepsilon))
	\mspace{7mu} \textrm{(by \hyperlink{vd}{$\on{VD}$} and \eqref{e:mi3})}\\
&= m_{a}\Bigl(\bigcup\nolimits_{x\in N_{\varepsilon}}B(x,\varepsilon)\Bigr)
	\leq m_{a}(K_{\varepsilon}) \xrightarrow{\varepsilon \downarrow 0} m_{a}(K)=0,
\end{align*}
which is a contradiction and thereby proves that $m_{s}(X)=0$.
\end{proof}

As the last step of the proof of Theorem \ref{t:main}-(b), we now establish first
the finiteness of $d_{\on{int}}$, and then the bi-Lipschitz equivalence of $d_{\on{int}}$
to $d$ under the additional assumption of the chain condition for $(X,d)$.

\begin{proposition} \label{p:bilip}
Let $(X,d,m,\mathcal{E},\mathcal{F})$ be a MMD space satisfying \hyperlink{vd}{$\on{VD}$},
\hyperlink{pi}{$\on{PI}(\Psi)$} and \hyperlink{cs}{$\on{CS}(\Psi)$}, and assume further
that $\Psi$ satisfies \eqref{e:case-ac}. Then $d_{\on{int}}$ is a geodesic metric on $X$.
Moreover, if additionally $(X,d)$ satisfies the chain condition, then
$d_{\on{int}}$ is bi-Lipschitz equivalent to $d$.
\end{proposition}

\begin{proof}
By Lemmas \ref{l:ge2}, \ref{l:pi-dist} and \ref{l:dist}, there exist $r_{0}>0$ and $C \geq 1$ such that
\begin{equation} \label{e:bl1}
C^{-1} d(x,y) \leq d_{\on{int}}(x,y) \leq C d(x,y)
	\mspace{13mu} \textrm{for all $x,y \in X$ with $d(x,y) \wedge d_{\on{int}}(x,y) < r_{0}$.}
\end{equation}
Let $d_{\varepsilon}$ and $d_{\on{int},\varepsilon}$ denote the $\varepsilon$-chain metric
corresponding to $d$ and $d_{\on{int}}$ respectively, as defined in Definition \ref{d:chain}-(a)
for each $\varepsilon > 0$; note that $d_{\on{int},\varepsilon}$ can be defined by \eqref{e:depsilon}
even though $d_{\on{int}}$ is yet to be shown to be a metric on $X$. Let $\varepsilon \in (0,r_{0})$. Then
we easily see from \eqref{e:depsilon}, \eqref{e:bl1} and the triangle inequality for $d$ and $d_{\on{int}}$
that for all $x,y \in X$,
\begin{equation} \label{e:bl2}
C^{-1}d(x,y) \leq \bigl(C^{-1} d_{C\varepsilon}(x,y)\bigr) \vee d_{\on{int}}(x,y)
	\leq d_{\on{int},\varepsilon}(x,y) \leq C d_{C^{-1}\varepsilon}(x,y)<\infty,
\end{equation}
where we used the fact that $d_{C^{-1}\varepsilon}(x,y)<\infty$ by \cite[Lemma 2.2]{Mur}.
It follows from \eqref{e:bl2}, \eqref{e:bl1} and the completeness of $(X,d)$ that
$d_{\on{int}}$ is a complete metric on $X$ compatible with the original topology of
$(X,d)$, and thus we can apply \cite[Theorem 1]{Stu95b} to obtain the geodesic property
of $d_{\on{int}}$, which together with \eqref{e:depsilon} and \eqref{e:bl2} implies that
\begin{equation} \label{e:bl3}
d_{\on{int}}(x,y) = d_{\on{int},\varepsilon}(x,y) \geq C^{-1}d(x,y) \qquad \textrm{for all $x,y \in X$.}
\end{equation}

Finally, assuming now that $(X,d)$ satisfies the chain condition, for some $C'\geq 1$
we have $d_{C^{-1}\varepsilon}(x,y)\leq C' d(x,y)$ for all $x,y \in X$,
which in combination with \eqref{e:bl2} shows that
\begin{equation} \label{e:bl4}
d_{\on{int},\varepsilon}(x,y) \leq C d_{C^{-1}\varepsilon}(x,y) \leq CC' d(x,y)
	\qquad \textrm{for all $x,y \in X$.}
\end{equation}
We therefore conclude from \eqref{e:bl3} and \eqref{e:bl4} the bi-Lipschitz
equivalence of $d_{\on{int}}$ to $d$.
\end{proof}

\begin{proof}[Proof of Theorem \textup{\ref{t:main}-(b)}]
We have \eqref{e:ge2} by Lemma \ref{l:ge2}, then \eqref{e:bi-Lipschitz-d-dint} by
\eqref{e:ge2}, Lemmas \ref{l:pi-dist} and \ref{l:dist}, and $m$ is a minimal energy-dominant
measure of $(\mathcal{E},\mathcal{F})$ by Propositions \ref{p:dom} and \ref{p:min}. Finally
by Proposition \ref{p:bilip}, $d_{\on{int}}$ is a geodesic metric on $X$, and it is bi-Lipschitz
equivalent to $d$ under the additional assumption of the chain condition for $(X,d)$.
\end{proof}

\section{Examples: Scale irregular Sierpi\'{n}ski gaskets} \label{s:sisg}

This section is devoted to presenting an application of Theorem \ref{t:main}-(a)
to a class of fractals called \emph{scale irregular Sierpi\'{n}ski gaskets},
which are constructed in a way similar to the standard Sierpi\'{n}ski gasket
($K^{2}$ in Figure \ref{fig:SGs}) but allowing different configurations of the cells
in different scales and thus are not exactly self-similar. This class of fractals
are also called \emph{homogeneous random Sierpi\'{n}ski gaskets} in the literature,
especially when the sequence of cell configurations in different scales is randomly
chosen according to some probability distribution, but here we prefer not to use
this term because we do not make such random construction.
We could introduce an abstract class of self-similar fractals generalizing the
Sierpi\'{n}ski gasket and use them to construct our scale irregular Sierpi\'{n}ski gaskets,
as is done in \cite{BH,Ham00} and \cite[Chapter 24]{Kig12}. For the sake of brevity,
however, we instead consider just a concrete family of self-similar Sierpi\'{n}ski
gaskets, which give rise to the higher dimensional analogs of the $2$-dimensional scale
irregular Sierpi\'{n}ski gaskets considered initially by Hambly in \cite{Ham92}.
\begin{figure}[b]\centering
\includegraphics[height=110pt]{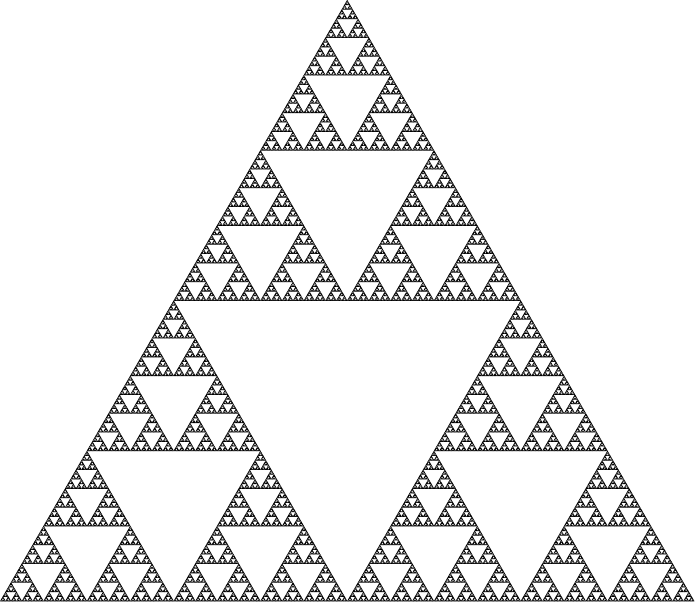}\quad
\includegraphics[height=110pt]{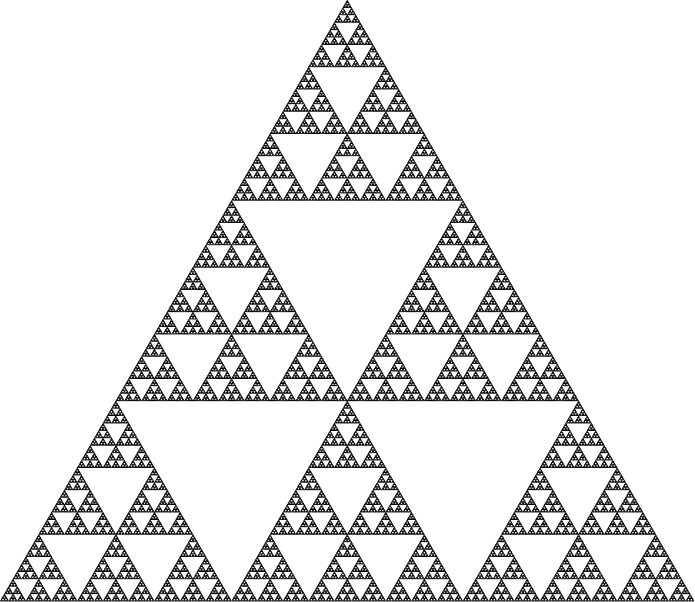}\quad
\includegraphics[height=110pt]{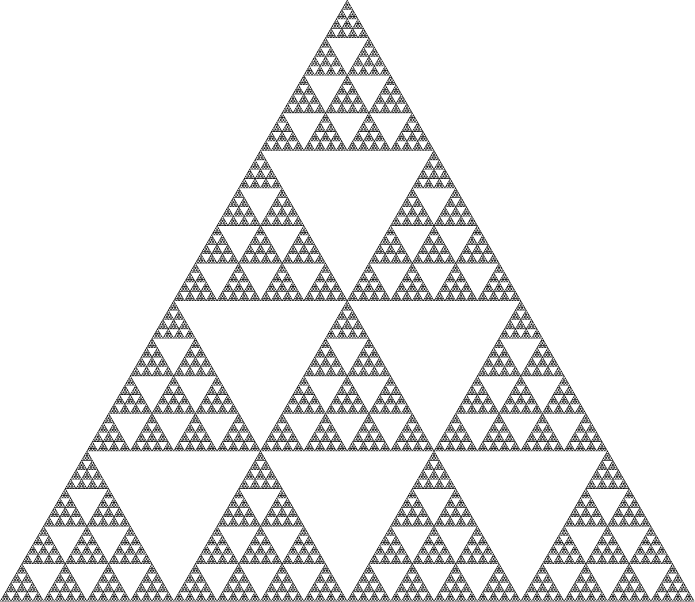}
\caption{The $2$-dimensional level-$l$ (self-similar) Sierpi\'{n}ski gaskets $K^{l}$ ($l=2,3,4$)}
\label{fig:SGs}
\end{figure}

Throughout this section, we fix $N\in\mathbb{N}\setminus\{1\}$ and a regular
$N$-dimensional simplex $\triangle\subset\mathbb{R}^{N}$ with side length $1$ and the set
of its vertices $\{q_{k}\mid k\in\{0,\ldots,N\}\}=:V_{0}$, where $\triangle$ denotes the convex
hull of $V_{0}$ in $\mathbb{R}^{N}$ and is thus a compact convex subset of $\mathbb{R}^{N}$.
For each $l\in\mathbb{N}\setminus\{1\}$, we set
$S_{l}:=\bigl\{(i_{k})_{k=1}^{N}\in(\mathbb{N}\cup\{0\})^{N}\bigm|\sum_{k=1}^{N}i_{k}\leq l-1\bigr\}$,
and for each $i=(i_{k})_{k=1}^{N}\in S_{l}$ set
$q^{l}_{i}:=q_{0}+\sum_{k=1}^{N}(i_{k}/l)(q_{k}-q_{0})$ and define
$F^{l}_{i}:\mathbb{R}^{N}\to\mathbb{R}^{N}$ by $F^{l}_{i}(x):=q^{l}_{i}+l^{-1}(x-q_{0})$.

Let $\bm{l}=(l_{n})_{n=1}^{\infty}\in(\mathbb{N}\setminus\{1\})^{\mathbb{N}}$ satisfy
$\sup_{n\in\mathbb{N}}l_{n}<\infty$, set $W^{\bm{l}}_{n}:=\prod_{k=1}^{n}S_{l_{k}}$ for each
$n\in\mathbb{N}$ and $F^{\bm{l}}_{w}:=F^{l_{1}}_{w_{1}}\circ\cdots\circ F^{l_{n}}_{w_{n}}$
for each $n\in\mathbb{N}$ and $w=w_{1}\ldots w_{n}\in W^{\bm{l}}_{n}$. We define the
\emph{$N$-dimensional level-$\bm{l}$ scale irregular Sierpi\'{n}ski gasket} $K^{\bm{l}}$
as the non-empty compact subset of $\triangle$ given by
\begin{equation}\label{e:sisg}
K^{\bm{l}}:=\bigcap_{n=1}^{\infty}\bigcup_{w\in W^{\bm{l}}_{n}}F^{\bm{l}}_{w}(\triangle)
\end{equation}
(see Figure \ref{fig:SISG}); note that
$\bigl\{\bigcup_{w\in W^{\bm{l}}_{n}}F^{\bm{l}}_{w}(\triangle)\bigr\}_{n=1}^{\infty}$
is a strictly decreasing sequence of non-empty compact subsets of $\triangle$ and that
\begin{equation}\label{e:sisg-boundary}
F^{\bm{l}}_{w}(\triangle)\cap F^{\bm{l}}_{v}(\triangle)=F^{\bm{l}}_{w}(V_{0})\cap F^{\bm{l}}_{v}(V_{0})
	\quad\textrm{for any $n\in\mathbb{N}$ and any $w,v\in W^{\bm{l}}_{n}$ with $w\not=v$.}
\end{equation}
We also set $V^{\bm{l}}_{0}:=V_{0}$ and
$V^{\bm{l}}_{n}:=\bigcup_{w\in W^{\bm{l}}_{n}}F^{\bm{l}}_{w}(V_{0})$ for each $n\in\mathbb{N}$,
so that $\{V^{\bm{l}}_{n}\}_{n=0}^{\infty}$ is a strictly increasing sequence of finite
subsets of $K^{\bm{l}}$ and $\bigcup_{n=0}^{\infty}V^{\bm{l}}_{n}$ is dense in $K^{\bm{l}}$.
In particular, for each $l\in\mathbb{N}\setminus\{1\}$ we let $\bm{l}_{l}:=(l)_{n=1}^{\infty}$
denote the constant sequence with value $l$, set $K^{l}:=K^{\bm{l}_{l}}$ and
$V^{l}_{n}:=V^{\bm{l}_{l}}_{n}$ for $n\in\mathbb{N}\cup\{0\}$, and call $K^{l}$ the
\emph{$N$-dimensional level-$l$ Sierpi\'{n}ski gasket}, which is exactly self-similar
in the sense that $K^{l}=\bigcup_{i\in S_{l}}F^{l}_{i}(K^{l})$ (see Figure \ref{fig:SGs}).
\begin{figure}[t]\centering
\includegraphics[height=300pt]{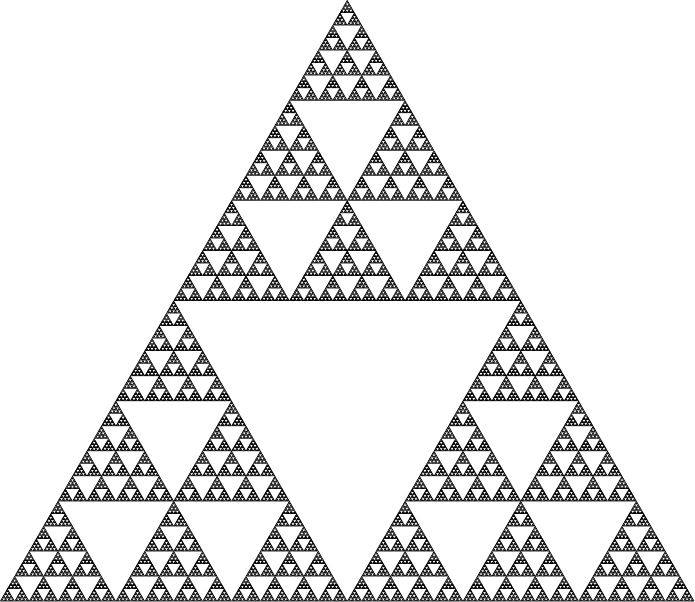}
\caption{A $2$-dimensional level-$\bm{l}$ scale irregular Sierpi\'{n}ski gasket $K^{\bm{l}}$ ($\bm{l}=(2,3,4,2,\ldots)$)}
\label{fig:SISG}
\end{figure}

As discussed in \cite{Ham92,BH,Ham00} (see also \cite[Part 4]{Kig12}),
we can define a canonical MMD space
$(K^{\bm{l}},d_{\bm{l}},m_{\bm{l}},\mathcal{E}^{\bm{l}},\mathcal{F}_{\bm{l}})$
over $K^{\bm{l}}$ with the metric $d_{\bm{l}}$ geodesic, as follows.
First, we define $d_{\bm{l}}:K^{\bm{l}}\times K^{\bm{l}}\to[0,\infty)$ by
\begin{equation}\label{e:sisg-metric}
d_{\bm{l}}(x,y):=\inf\{\on{Length}(\gamma)\mid
	\textrm{$\gamma:[0,1]\to K^{\bm{l}}$, $\gamma$ is continuous, $\gamma(0)=x$, $\gamma(1)=y$}\},
\end{equation}
where $\on{Length}(\gamma)$ denotes the Euclidean length of $\gamma$, i.e., the total
variation of $\gamma$ as an $\mathbb{R}^{N}$-valued map. Then it is easy to see,
by following \cite[Proof of Lemma 2.4]{BH}, that
$d_{\bm{l}}$ is a geodesic metric on $K^{\bm{l}}$ which is bi-Lipschitz equivalent
to the restriction to $K^{\bm{l}}$ of the Euclidean metric on $\mathbb{R}^{N}$.
Next, the standard measure-theoretic arguments immediately show that there exists
a unique Borel probability measure $m_{\bm{l}}$ on $K^{\bm{l}}$ such that
\begin{equation}\label{e:sisg-measure}
m_{\bm{l}}\bigl(F^{\bm{l}}_{w}(K^{\bm{l}})\bigr)=\frac{1}{M^{\bm{l}}_{n}}
	\qquad\textrm{for any $n\in\mathbb{N}$ and any $w\in W^{\bm{l}}_{n}$,}
\end{equation}
where $M^{\bm{l}}_{n}:=(\# S_{l_{1}})\cdots(\# S_{l_{n}})$,
and then $m_{\bm{l}}$ is clearly a Radon measure on $K^{\bm{l}}$ with full support.
The measure $m_{\bm{l}}$ can be considered as the ``uniform distribution on $K^{\bm{l}}$''.

The Dirichlet form $(\mathcal{E}^{\bm{l}},\mathcal{F}_{\bm{l}})$ is constructed as the
``inductive limit'' of a certain canonical sequence of discrete Dirichlet forms on the
finite sets $\{V^{\bm{l}}_{n}\}_{n=0}^{\infty}$ by the standard method presented in
\cite[Chapter 3]{Kig01} (see also \cite[Sections 6 and 7]{Bar98}).
We start with defining a non-negative definite symmetric bilinear form
$\mathcal{E}^{0}:\mathbb{R}^{V_{0}}\times\mathbb{R}^{V_{0}}\to\mathbb{R}$
on $\mathbb{R}^{V_{0}}=\mathbb{R}^{V^{\bm{l}}_{0}}$ by
\begin{equation}\label{e:sisg-form-0}
\mathcal{E}^{0}(f,g)
	:=\frac{1}{2}\sum_{j,k=0}^{N}\bigl(f(q_{j})-f(q_{k})\bigr)\bigl(g(q_{j})-g(q_{k})\bigr),
	\qquad f,g\in\mathbb{R}^{V_{0}}.
\end{equation}
We would like to define a bilinear form $\mathcal{E}^{\bm{l},n}$ on
$\mathbb{R}^{V^{\bm{l}}_{n}}$ for each $n\in\mathbb{N}$ as the sum of the copies of
\eqref{e:sisg-form-0} on $\{F^{\bm{l}}_{w}(V_{0})\}_{w\in W^{\bm{l}}_{n}}$ and
take their limit as $n\to\infty$, but for the existence of their limit they
actually need to be multiplied by certain scaling factors given as follows.
For each $l\in\mathbb{N}\setminus\{1\}$, the Euclidean-geometric symmetry of
$V_{0}=V^{l}_{0}$ and $V^{l}_{1}$ immediately implies the existence of a unique
$r_{l}\in(0,\infty)$ such that for any $f\in\mathbb{R}^{V_{0}}$,
\begin{equation}\label{e:sisg-rl}
\min\Biggl\{\sum_{i\in S_{l}}\mathcal{E}^{0}\bigl(g\circ F^{l}_{i}|_{V_{0}},g\circ F^{l}_{i}|_{V_{0}}\bigr)
	\Biggm|\textrm{$g\in\mathbb{R}^{V^{l}_{1}}$, $g|_{V_{0}}=f$}\Biggr\}
	=r_{l}\mathcal{E}^{0}(f,f),
\end{equation}
and $r_{l}\in(0,1)$ by \cite[Corollary 3.1.9]{Kig01}.
Then setting $\mathcal{E}^{\bm{l},0}:=\mathcal{E}^{0}$ and
defining for each $n\in\mathbb{N}$ a non-negative definite symmetric bilinear form
$\mathcal{E}^{\bm{l},n}:\mathbb{R}^{V^{\bm{l}}_{n}}\times\mathbb{R}^{V^{\bm{l}}_{n}}\to\mathbb{R}$
on $\mathbb{R}^{V^{\bm{l}}_{n}}$ by
\begin{equation}\label{e:sisg-form-ln}
\mathcal{E}^{\bm{l},n}(f,g)
	:=\frac{1}{R^{\bm{l}}_{n}}\sum_{w\in W^{\bm{l}}_{n}}\mathcal{E}^{0}\bigl(f\circ F^{\bm{l}}_{w}|_{V_{0}},g\circ F^{\bm{l}}_{w}|_{V_{0}}\bigr),
	\qquad f,g\in\mathbb{R}^{V^{\bm{l}}_{n}},
\end{equation}
where $R^{\bm{l}}_{n}:=r_{l_{1}}\cdots r_{l_{n}}$, we easily see from \eqref{e:sisg-rl} and
\eqref{e:sisg-boundary} that for any $n\in\mathbb{N}$ and any $f\in\mathbb{R}^{V^{\bm{l}}_{n-1}}$,
\begin{equation}\label{e:sisg-form-compatible}
\min\bigl\{\mathcal{E}^{\bm{l},n}(g,g)\bigm|\textrm{$g\in\mathbb{R}^{V^{\bm{l}}_{n}}$, $g|_{V^{\bm{l}}_{n-1}}=f$}\bigr\}
	=\mathcal{E}^{\bm{l},n-1}(f,f).
\end{equation}
The equality \eqref{e:sisg-form-compatible} allows us to take the
``inductive limit'' of $\{\mathcal{E}^{\bm{l},n}\}_{n=0}^{\infty}$, i.e.,
to define a linear subspace $\mathcal{F}_{\bm{l}}$ of $\mathcal{C}(K^{\bm{l}})$
and a non-negative definite symmetric bilinear form
$\mathcal{E}^{\bm{l}}:\mathcal{F}_{\bm{l}}\times\mathcal{F}_{\bm{l}}\to\mathbb{R}$
on $\mathcal{F}_{\bm{l}}$ by
\begin{align}\label{e:sisg-DF-domain}
\mathcal{F}_{\bm{l}}&:=\Bigl\{f\in\mathcal{C}(K^{\bm{l}})\Bigm|\lim_{n\to\infty}\mathcal{E}^{\bm{l},n}(f|_{V^{\bm{l}}_{n}},f|_{V^{\bm{l}}_{n}})<\infty\Bigr\},\\
\mathcal{E}^{\bm{l}}(f,g)&:=\lim_{n\to\infty}\mathcal{E}^{\bm{l},n}(f|_{V^{\bm{l}}_{n}},g|_{V^{\bm{l}}_{n}})\in\mathbb{R},
	\quad f,g\in\mathcal{F}_{\bm{l}},
\label{e:sisg-DF-form}
\end{align}
where $\bigl\{\mathcal{E}^{\bm{l},n}(f|_{V^{\bm{l}}_{n}},f|_{V^{\bm{l}}_{n}})\bigr\}_{n=0}^{\infty}\subset[0,\infty)$
is non-decreasing by \eqref{e:sisg-form-compatible} and hence has a limit in $[0,\infty]$
for any $f\in\mathcal{C}(K^{\bm{l}})$. Then exactly the same arguments as in
\cite[Chapter 22]{Kig12} show that $(\mathcal{E}^{\bm{l}},\mathcal{F}_{\bm{l}})$ is a local
regular resistance form on $K^{\bm{l}}$ in the sense of \cite[Chapters 3, 6 and 7]{Kig12}
with its resistance metric giving the same topology as $d_{\bm{l}}$, and is thereby
a strongly local, regular symmetric Dirichlet form on $L^{2}(K^{\bm{l}},m_{\bm{l}})$
by \cite[Theorem 9.4]{Kig12}.

For the present MMD space
$(K^{\bm{l}},d_{\bm{l}},m_{\bm{l}},\mathcal{E}^{\bm{l}},\mathcal{F}_{\bm{l}})$,
it turns out that the right choice of a space-time scale function $\Psi$
is the homeomorphism $\Psi_{\bm{l}}:[0,\infty)\to[0,\infty)$ defined by
\begin{equation} \label{e:sisg-Psi}
\Psi_{\bm{l}}(s):=
	\begin{cases}
		\dfrac{(L^{\bm{l}}_{n}s)^{\beta_{l_{n}}}}{T^{\bm{l}}_{n}} & \textrm{if $n\in\mathbb{N}$ and $s\in[(L^{\bm{l}}_{n})^{-1},(L^{\bm{l}}_{n-1})^{-1}]$,} \\
		s^{\beta^{\min}_{\bm{l}}} & \textrm{if $s\in[1,\infty)$,}
	\end{cases}
\end{equation}
where $\beta_{l}:=\log_{l}(\# S_{l}/r_{l})$ for $l\in\mathbb{N}\setminus\{1\}$,
$\beta^{\min}_{\bm{l}}:=\min_{n\in\mathbb{N}}\beta_{l_{n}}$,
$L^{\bm{l}}_{0}:=T^{\bm{l}}_{0}:=1$, $L^{\bm{l}}_{n}:=l_{1}\cdots l_{n}$ and
$T^{\bm{l}}_{n}:=M^{\bm{l}}_{n}/R^{\bm{l}}_{n}$ for $n\in\mathbb{N}$, so that
$\beta_{l}\in(1,\infty)$ for any $l\in\mathbb{N}\setminus\{1\}$ by $\# S_{l}\geq l+1$ and $r_{l}<1$
and hence also $\beta^{\min}_{\bm{l}}\in(1,\infty)$ by $\sup_{n\in\mathbb{N}}l_{n}<\infty$.
It is immediate from \eqref{e:sisg-Psi} and $\sup_{n\in\mathbb{N}}l_{n}<\infty$ that
$\Psi_{\bm{l}}$ satisfies Assumption \ref{a:reg} with $\beta^{\min}_{\bm{l}}$
and $\beta^{\max}_{\bm{l}}:=\max_{n\in\mathbb{N}}\beta_{l_{n}}$ in place of
$\beta_{0}$ and $\beta_{1}$, respectively. In particular, if $l\in\mathbb{N}\setminus\{1\}$
and $\bm{l}$ is the constant sequence $\bm{l}_{l}=(l)_{n=1}^{\infty}$ with value $l$,
then $\Psi_{\bm{l}_{l}}(s)=s^{\beta_{l}}$ for any $s\in[0,\infty)$.

The following result is essentially a special case of \cite[Theorem 4.5 and Lemma 5.3]{BH},
and it is concluded from \cite[Theorem 15.10]{Kig12} by proving the conditions
$\on{(DM1)}_{\Psi_{\bm{l}},d_{\bm{l}}}$ and $\on{(DM2)}_{\Psi_{\bm{l}},d_{\bm{l}}}$
defined in \cite[Definition 15.9-(3),(4)]{Kig12},
which can be achieved in exactly the same way as \cite[Chapter 24]{Kig12}.

\begin{theorem}\label{t:sisg-HKE}
$(K^{\bm{l}},d_{\bm{l}},m_{\bm{l}},\mathcal{E}^{\bm{l}},\mathcal{F}_{\bm{l}})$
satisfies \hyperlink{vd}{$\on{VD}$} and \hyperlink{hke}{$\on{HKE}(\Psi_{\bm{l}})$}.
\end{theorem}

\begin{corollary}\label{c:sisg-VDPICS}
$(K^{\bm{l}},d_{\bm{l}},m_{\bm{l}},\mathcal{E}^{\bm{l}},\mathcal{F}_{\bm{l}})$
satisfies \hyperlink{vd}{$\on{VD}$}, \hyperlink{pi}{$\on{PI}(\Psi_{\bm{l}})$}
and \hyperlink{cs}{$\on{CS}(\Psi_{\bm{l}})$}.
\end{corollary}

\begin{proof}
This is immediate from Assumption \ref{a:reg} for $\Psi_{\bm{l}}$,
Theorems \ref{t:sisg-HKE} and \ref{t:HKE-VDPICS}.
\end{proof}

Thus our present MMD space
$(K^{\bm{l}},d_{\bm{l}},m_{\bm{l}},\mathcal{E}^{\bm{l}},\mathcal{F}_{\bm{l}})$
will prove to fall into the situation of Theorem \ref{t:main}-(a) once $\Psi_{\bm{l}}$
has been shown to satisfy \eqref{e:case-sing}, which is indeed the case as stated in
Proposition \ref{p:sisg-case-sing} below. Note that this proposition is not
entirely obvious since it seems impossible to calculate the values of $r_{l}$
and $\beta_{l}$ explicitly for general $l\in\mathbb{N}\setminus\{1\}$.

\begin{proposition}\label{p:sisg-case-sing}
$\beta_{l}>2$ for any $l\in\mathbb{N}\setminus\{1\}$. In particular,
$\beta^{\min}_{\bm{l}}>2$ and $\Psi_{\bm{l}}$ satisfies \eqref{e:case-sing}.
\end{proposition}

\begin{proof}
Let $l\in\mathbb{N}\setminus\{1\}$, consider the case where $\bm{l}$ is the constant
sequence $\bm{l}_{l}=(l)_{n=1}^{\infty}$ with value $l$, i.e., that of the $N$-dimensional
level-$l$ Sierpi\'{n}ski gasket $K^{l}$, set $d_{l}:=d_{\bm{l}_{l}}$, $m_{l}:=m_{\bm{l}_{l}}$
and $(\mathcal{E}^{l},\mathcal{F}_{l}):=(\mathcal{E}^{\bm{l}_{l}},\mathcal{F}_{\bm{l}_{l}})$
and let $\Gamma_{l}(f,f)$ denote the energy measure of $f\in\mathcal{F}_{l}$
associated with $(K^{l},d_{l},m_{l},\mathcal{E}^{l},\mathcal{F}_{l})$.
Then by \cite[Theorem 2]{HN} we have $\Gamma_{l}(f,f) \perp m_{l}$
for all $f\in\mathcal{F}_{l}$, which together with Corollary \ref{c:sisg-VDPICS}
for $\bm{l}=\bm{l}_{l}$ and Theorem \ref{t:main}-(b) implies that
$\limsup_{s\downarrow 0}s^{\beta_{l}-2}=\limsup_{s\downarrow 0}s^{-2}\Psi_{\bm{l}_{l}}(s)=0$
since only one of $\Gamma_{l}(f,f) \perp m_{l}$ and $\Gamma_{l}(f,f) \ll m_{l}$
can hold for each $f\in\mathcal{F}_{l}\setminus\mathbb{R}\one_{K^{l}}$
by $\Gamma_{l}(f,f)(K^{l})=\mathcal{E}^{l}(f,f)>0$. Thus $\beta_{l}>2$,
which in combination with $\sup_{n\in\mathbb{N}}l_{n}<\infty$ yields $\beta^{\min}_{\bm{l}}>2$.
Now \eqref{e:reg} for $\Psi_{\bm{l}}$ with $\beta^{\min}_{\bm{l}}>2$
in place of $\beta_{0}$ shows \eqref{e:case-sing} for $\Psi_{\bm{l}}$.

An alternative elementary proof of $\beta_{l}>2$, which is a slight modification of
that suggested by an anonymous referee, is based on the specific structure of $V^{l}_{1}$
and $\mathcal{E}^{0}$ and goes as follows. Define $f:\mathbb{R}^{N}\to\mathbb{R}$
by $f\bigl(q_{0}+\sum_{k=1}^{N}a_{k}(q_{k}-q_{0})\bigr):=\sum_{k=1}^{N}a_{k}$
for each $(a_{k})_{k=1}^{N}\in\mathbb{R}^{N}$. Then since 
$f\circ F^{l}_{i}=l^{-1}f+f(q^{l}_{i})\one_{\mathbb{R}^{N}}$ for any $i\in S_{l}$ and
$g:=f|_{V^{l}_{1}}$ is easily seen not to attain the minimum in the left-hand side of
\eqref{e:sisg-rl} with $f|_{V_{0}}$ in place of $f$, from \eqref{e:sisg-form-0} and
\eqref{e:sisg-rl} we obtain
\begin{equation*}
0<\mathcal{E}^{0}(f|_{V_{0}},f|_{V_{0}})
	<\frac{1}{r_{l}}\sum_{i\in S_{l}}\mathcal{E}^{0}\bigl(f\circ F^{l}_{i}|_{V_{0}},f\circ F^{l}_{i}|_{V_{0}}\bigr)
	=\frac{\# S_{l}}{r_{l}}l^{-2}\mathcal{E}^{0}(f|_{V_{0}},f|_{V_{0}}),
\end{equation*}
whence $\# S_{l}/r_{l}>l^{2}$ and $\beta_{l}=\log_{l}(\# S_{l}/r_{l})>2$.
\end{proof}

\begin{remark}\label{rmk:sisg-case-sing}
The alternative proof of $\beta_{l}>2$ in the second paragraph of the proof of
Proposition \ref{p:sisg-case-sing} above can be adapted to give an elementary
proof of the counterpart of $\beta_{l}>2$ for the canonical Dirichlet form
on Sierpi\'{n}ski carpets; see \cite{Kaj20a} for details.
\end{remark}

Finally, applying Theorem \ref{t:main}-(a) to
$(K^{\bm{l}},d_{\bm{l}},m_{\bm{l}},\mathcal{E}^{\bm{l}},\mathcal{F}_{\bm{l}})$
on the basis of Corollary \ref{c:sisg-VDPICS} and Proposition \ref{p:sisg-case-sing},
we arrive at the following result.

\begin{theorem}\label{t:sisg-sing}
Let $\Gamma_{\bm{l}}(f,f)$ denote the energy measure of
$f \in \mathcal{F}_{\bm{l}}$ associated with the MMD space
$(K^{\bm{l}},d_{\bm{l}},m_{\bm{l}},\mathcal{E}^{\bm{l}},\mathcal{F}_{\bm{l}})$.
Then $\Gamma_{\bm{l}}(f,f) \perp m_{\bm{l}}$ for all $f \in \mathcal{F}_{\bm{l}}$.
\end{theorem}

%
%

\begin{appendix}

\section{Miscellaneous facts}\label{sec:appendix}

In this appendix, we state and prove a couple of miscellaneous facts utilized in the proof
of Theorem \ref{t:main}-(a). The former (Proposition \ref{p:cc}) achieves the equivalence
between the chain condition and the bi-Lipschitz equivalence to a geodesic metric and
allows us to reduce the proof to the case where the metric is geodesic.
The latter (Proposition \ref{p:rd}) is a straightforward extension, to a general
metric measure space satisfying \hyperlink{vd}{$\on{VD}$}, of the classical Lebesgue
differentiation theorem \cite[Theorem 7.13]{Rud} for singular measures on the Euclidean space,
and here we give a complete proof of it for the reader's convenience.

\subsection{Chain condition and bi-Lipschitz equivalence to a geodesic metric} \label{ssec:chain-geodesic}

\begin{proposition} \label{p:cc}
Let $(X,d)$ be a metric space such that $B(x,r):=\{y\in X\mid d(x,y)<r\}$ is relatively
compact in $X$ for any $(x,r)\in X\times(0,\infty)$. Then the following are equivalent:
\begin{itemize}[label=\textup{(b)},align=left,leftmargin=*]
\item[\textup{(a)}]$(X,d)$ satisfies the chain condition.
\item[\textup{(b)}]There exists a geodesic metric $\rho$ on $X$ which is bi-Lipschitz equivalent
	to $d$, i.e., satisfies $C^{-1} d(x,y) \le \rho(x,y) \le C d(x,y)$
	for any $x,y \in X$ for some $C\in[1,\infty)$.
\end{itemize}
\end{proposition}

We need the following definition and lemma for the proof of Proposition \ref{p:cc}.

\begin{definition} \label{d:mp}
Let $(X,d)$ be a metric space and let $x,y\in X$. We say that $z\in X$ is a
\emph{midpoint} in $(X,d)$ between $x,y$ if $d(x,z)=d(y,z)=d(x,y)/2$.
\end{definition}

\begin{lemma} \label{l:mp}
Let $(X,d)$ be a metric space. If $\varepsilon>0$ and $x,y \in X$ satisfy
$d_{\varepsilon}(x,y)<\infty$, then there exists $z \in X$ such that
$|2d_{\varepsilon}(x,z)-d_{\varepsilon}(x,y)| \leq 5\varepsilon$ and
$|2d_{\varepsilon}(y,z)-d_{\varepsilon}(x,y)| \leq 5\varepsilon$.
\end{lemma}

\begin{proof}
By the definition \eqref{e:depsilon} of $d_{\varepsilon}(x,y)$ and the assumption
$d_{\varepsilon}(x,y)<\infty$ we can take an $\varepsilon$-chain $\{x_{i}\}_{i=0}^{n}$
in $(X,d)$ from $x$ to $y$ such that
\begin{equation} \label{e:mp1}
\sum_{i=0}^{n-1}d(x_{i},x_{i+1}) \geq d_{\varepsilon}(x,y)
	\geq \sum_{i=0}^{n-1}d(x_{i},x_{i+1}) - \varepsilon.
\end{equation}
Let $k\in\{1,\ldots,n\}$ be the smallest integer such that
\begin{equation}\label{e:mp2}
\sum_{i=0}^{k-1}d(x_{i},x_{i+1}) \geq \frac{1}{2}\sum_{i=0}^{n-1}d(x_{i},x_{i+1}).
\end{equation}
We claim that $z:=x_{k}$ satisfies the desired inequalities.
Indeed, by $d(x_{k-1},x_{k})<\varepsilon$ and the minimality of $k$
among the elements of $\{1,\ldots,n\}$ with the property \eqref{e:mp2}, we have
\begin{equation} \label{e:mp3}
\sum_{i=0}^{k-1}d(x_{i},x_{i+1}) \geq \frac{1}{2}\sum_{i=0}^{n-1}d(x_{i},x_{i+1})
	> \sum_{i=0}^{k-1}d(x_{i},x_{i+1}) -\varepsilon
\end{equation}
and
\begin{equation}\label{e:mp4}
\frac{1}{2}\sum_{i=0}^{n-1}d(x_{i},x_{i+1}) \geq \sum_{i=k}^{n-1}d(x_{i},x_{i+1})
	> \frac{1}{2}\sum_{i=0}^{n-1}d(x_{i},x_{i+1}) -\varepsilon.
\end{equation}
Noting that $d_{\varepsilon}$ satisfies the triangle inequality, we see from the lower
inequality in \eqref{e:mp1} and the definition \eqref{e:depsilon} of $d_{\varepsilon}$ that
\begin{equation*}
\sum_{i=0}^{n-1}d(x_{i},x_{i+1}) - \varepsilon
	\leq d_{\varepsilon}(x,y)\le d_{\varepsilon}(x,z)+d_{\varepsilon}(y,z)
	\leq \sum_{i=0}^{n-1}d(x_{i},x_{i+1}),
\end{equation*}
which yields
\begin{equation} \label{e:mp5}
-\varepsilon
	\leq \Biggl(d_{\varepsilon}(x,z)-\sum_{i=0}^{k-1} d(x_{i},x_{i+1}) \Biggr)
		+\Biggl(d_{\varepsilon}(y,z)-\sum_{i=k}^{n-1} d(x_{i},x_{i+1}) \Biggr) \leq 0.
\end{equation}
Since both of the terms in \eqref{e:mp5} are non-positive by \eqref{e:depsilon}, we obtain
\begin{equation} \label{e:mp6}
\Biggl|d_{\varepsilon}(x,z)-\sum_{i=0}^{k-1} d(x_{i},x_{i+1})\Biggr| \leq \varepsilon
	\qquad\textrm{and}\qquad
	\Biggl|d_{\varepsilon}(y,z)-\sum_{i=k}^{n-1} d(x_{i},x_{i+1})\Biggr| \leq \varepsilon.
\end{equation}
Now it follows from the triangle inequality, \eqref{e:mp6}, \eqref{e:mp3} and \eqref{e:mp1} that
\begin{align*}
\biggl|d_{\varepsilon}(x,z)- \frac{1}{2} d_{\varepsilon}(x,y)\biggr|
	&\leq \Biggl|d_{\varepsilon}(x,z)\mspace{-0.4mu}-\mspace{-0.4mu}\sum_{i=0}^{k-1} d(x_{i},x_{i+1})\Biggr|
	+ \Biggl|\sum_{i=0}^{k-1} d(x_{i},x_{i+1})\mspace{-0.4mu}-\mspace{-0.4mu}\frac{1}{2} \sum_{i=0}^{n-1}d(x_{i},x_{i+1})\Biggr|\\
&\qquad+\frac{1}{2}\Biggl|d_{\varepsilon}(x,y)-\sum_{i=0}^{n-1}d(x_{i},x_{i+1})\Biggr|\\
&\leq \varepsilon+\varepsilon+\frac{\varepsilon}{2} = \frac{5}{2}\varepsilon,
\end{align*}
and in the same way from \eqref{e:mp6}, \eqref{e:mp4} and \eqref{e:mp1} that
$|2d_{\varepsilon}(y,z)-d_{\varepsilon}(x,y)|\leq 5\varepsilon$.
\end{proof}

\begin{proof}[Proof of Proposition \textup{\ref{p:cc}}]
$\textup{(b)}\mspace{-1mu}\Rightarrow\mspace{-1mu}\textup{(a)}$:
Let $\varepsilon >0$ and $x,y \in X$. Note that $\rho_{C^{-1}\varepsilon}(x,y)=\rho(x,y)$
by the definition \eqref{e:depsilon} of $\rho_{C^{-1}\varepsilon}(x,y)$ and the
geodesic property of $\rho$. Since $C^{-1} \rho \leq d \leq C \rho$ by (b), each
$C^{-1}\varepsilon$-chain in $(X,\rho)$ from $x$ to $y$ is also an $\varepsilon$-chain
in $(X,d)$ from $x$ to $y$, and therefore
\begin{equation*}
d_{\varepsilon}(x,y) \leq C \rho_{C^{-1}\varepsilon}(x,y) = C \rho(x,y) \leq C^{2} d(x,y).
\end{equation*}

\noindent
$\textup{(a)}\mspace{-1mu}\Rightarrow\mspace{-1mu}\textup{(b)}$:
Note that for each $x,y \in X$, $(0,\infty)\ni\varepsilon \mapsto d_{\varepsilon}(x,y)$ is a non-increasing
function and hence the limit $\rho(x,y):=\lim_{\varepsilon \downarrow 0}d_{\varepsilon}(x,y)$
exists. Since $d_{\varepsilon}$ is a metric on $X$ and $d \leq d_{\varepsilon} \leq C d$
for any $\varepsilon>0$ for some $C\geq 1$ by (a), $\rho$ is a metric on $X$, satisfies
$d \leq \rho \leq C d$ and is thus bi-Lipschitz equivalent to $d$, which in particular
yields the completeness of the metric space $(X,\rho)$, thanks to that of $(X,d)$ implied
by the assumed relative compactness of $B(x,r)$ in $X$ for all $(x,r)\in X\times(0,\infty)$.

It remains to prove that $(X,\rho)$ is geodesic, and by its completeness and
\cite[Proof of Theorem 2.4.16]{BBI} it suffices to show that for any $x,y \in X$
there exists a midpoint $z\in X$ in $(X,\rho)$ between $x,y$. To this end,
let $x,y \in X$ and, noting Lemma \ref{l:mp}, for each $n\in\mathbb{N}$
choose $z_{n}\in X$ so that
\begin{equation} \label{e:n-midpoint}
|2d_{n^{-1}}(x,z_{n})-d_{n^{-1}}(x,y)| \leq 5n^{-1}\quad\textrm{and}\quad
	|2d_{n^{-1}}(y,z_{n})-d_{n^{-1}}(x,y)| \leq 5n^{-1}.
\end{equation}
Then since $\{z_{n}\}_{n=1}^{\infty}$ is included in the relatively compact subset
$B(x,Cd(x,y)+5)$ of $X$ by \eqref{e:n-midpoint} and $d \leq d_{n^{-1}} \leq C d$, there
exists a subsequence $\{z_{n_{k}}\}_{k=1}^{\infty}$ of $\{z_{n}\}_{n=1}^{\infty}$
converging to some $z\in X$ in $(X,d)$. Now for any $k\in\mathbb{N}$, by the triangle
inequality for $d_{n_{k}^{-1}}$, \eqref{e:n-midpoint}, $d_{n_{k}^{-1}} \leq C d$
and $\lim_{j\to\infty}d(z,z_{n_{j}})=0$ we obtain
\begin{align*}
\bigl|2d_{n_{k}^{-1}}(x,z)-d_{n_{k}^{-1}}(x,y)\bigr|
	&\leq 2\bigl|d_{n_{k}^{-1}}(x,z)-d_{n_{k}^{-1}}(x,z_{n_{k}})\bigr|+\bigl|2d_{n_{k}^{-1}}(x,z_{n_{k}})-d_{n_{k}^{-1}}(x,y)\bigr|\\
&\leq 2d_{n_{k}^{-1}}(z,z_{n_{k}})+5n_{k}^{-1}\leq 2Cd(z,z_{n_{k}})+5n_{k}^{-1}\xrightarrow{k\to\infty}0,
\end{align*}
which yields $2\rho(x,z)-\rho(x,y)=\lim_{k\to\infty}\bigl(2d_{n_{k}^{-1}}(x,z)-d_{n_{k}^{-1}}(x,y)\bigr)=0$.
Exactly the same argument also shows $2\rho(y,z)-\rho(x,y)=0$, proving that $z$ is
a midpoint in $(X,\rho)$ between $x,y$ and thereby completing the proof.
\end{proof}

\subsection{Lebesgue's differentiation theorem for singular measures} \label{ssec:Lebesgue-diff-sing}
\begin{proposition}[cf.\ {\cite[Theorem 7.13]{Rud}}]\label{p:rd}
Let $(X,d,m)$ be a metric measure space satisfying \hyperlink{vd}{$\on{VD}$},
let $\nu$ be a Radon measure on $X$, i.e., a Borel measure on $X$ which is finite on
any compact subset of $X$, and assume $\nu \perp m$. Then
\begin{equation} \label{e:rd}
\lim_{r \downarrow 0} \frac{\nu(B(x,r))}{m(B(x,r))}=0 \qquad \textrm{for $m$-a.e.\ $x \in X$.}
\end{equation}
\end{proposition}

\begin{proof}
By taking $x_{0}\in X$ and considering $\nu(\cdot\cap B(x_{0},n))$ for each $n\in\mathbb{N}$
instead of $\nu$, we may assume without loss of generality that $\nu(X)<\infty$.
For each $x\in X$, we define
\begin{align*}
(Q_{r} \nu)(x) &:= \frac{\nu(B(x,r))}{m(B(x,r))}, \quad r\in(0,\infty), &
	\begin{split}
	(M \nu)(x) &:= \sup_{r\in(0,\infty)} (Q_{r} \nu)(x),\\
	(\overline{D}\nu)(x) &:= \limsup_{r \downarrow 0} (Q_{r} \nu)(x).
	\end{split}
\end{align*}
Since $(0,\infty) \ni r \mapsto m(B(x,r))$ and $(0,\infty) \ni r \mapsto \nu(B(x,r))$
are left-continuous, we have
\begin{equation} \label{e:rd1}
(M\nu)(x) = \sup_{r \in (0,\infty)\cap \mathbb{Q}} (Q_{r}\nu)(x)\quad\textrm{ and }\quad
	(\overline{D}\nu)(x) = \lim_{n \to \infty} \sup_{r \in(0,n^{-1})\cap \mathbb{Q}} (Q_r\nu)(x).
\end{equation}
An easy application of the triangle inequality shows that the functions
$X\ni x \mapsto m(B(x, r))$ and $X\ni x \mapsto \nu(B(x,r))$ are lower semi-continuous and
hence Borel measurable. Thus $X\ni x \mapsto (Q_{r}\nu)(x)$ is also Borel measurable and so
are $X\ni x \mapsto (M\nu)(x)$ and $X\ni x \mapsto (\overline{D}\nu)(x)$ by \eqref{e:rd1}.
Let $C_{D}$ denote the constant in \hyperlink{vd}{$\on{VD}$}. Using the estimate
$m(B(x,3r)) \leq C_{D}^{2} m(B(x,r))$ for $(x,r)\in X\times(0,\infty)$ and the arguments
in \cite[Proofs of Lemma 7.3 and Theorem 7.4]{Rud} together with the inner regularity
of $m$ (see, e.g., \cite[Theorem 2.18]{Rud}), we obtain the maximal inequality
\begin{equation} \label{e:rd3}
m\bigl((M\nu)^{-1}((\lambda,\infty])\bigr) \leq C_{D}^{2} \lambda^{-1} \nu(X)
	\qquad \textrm{for all $\lambda>0$.}
\end{equation}

Let $\lambda,\varepsilon>0$. Since $\nu \perp m$, the inner regularity of $\nu$
(see, e.g., \cite[Theorem 2.18]{Rud}) implies the existence of a compact subset $K$
of $X$ such that $m(K)=0$ and $\nu(K) > \nu(X) -\varepsilon$.
Set $\nu_{1} := \nu(\cdot \cap K)$ and $\nu_{2} := \nu(\cdot\cap(X\setminus K))$,
so that $\nu=\nu_{1}+\nu_{2}$ and $\nu_{2}(X) < \varepsilon$.
For every $x \in X \setminus K$, we have
\begin{equation*}
(\overline{D}\nu)(x) = (\overline{D}\nu_{2})(x) \leq (M \nu_{2})(x),
\end{equation*}
hence
\begin{equation*}
(\overline{D}\nu)^{-1}((\lambda,\infty]) \subset K \cup (M \nu_{2})^{-1}((\lambda,\infty]),
\end{equation*}
and therefore it follows from $m(K)=0$, \eqref{e:rd3} for the measure $\nu_{2}$
and $\nu_{2}(X) < \varepsilon$ that
\begin{equation} \label{e:rd4}
m\bigl((\overline{D}\nu)^{-1}((\lambda,\infty])\bigr)
	\leq m\bigl((M\nu_{2})^{-1}((\lambda,\infty])\bigr)
	\leq C_{D}^{2} \lambda^{-1} \nu_{2}(X)
	< C_{D}^{2} \lambda^{-1} \varepsilon.
\end{equation}
Since \eqref{e:rd4} holds for every $\lambda,\varepsilon>0$,
we conclude that $\overline{D}\nu=0$ $m$-a.e., which is \eqref{e:rd}.
\end{proof}

\end{appendix}

\section*{Acknowledgements}
The authors would like to thank Fabrice Baudoin for having pointed out to them that
the results in Proposition \ref{p:dom} and Remark \ref{rmk:dom} were obtained already
in \cite[Lemma 2.11]{ABCRST} with a very similar proof in a more restrictive situation.
They also would like to thank Martin T.\ Barlow for having suggested in \cite{Bar19}
considering the thin scale irregular Sierpi\'{n}ski gaskets in relation to Remark \ref{rmk:main}
and Conjecture \ref{conj:main}, and an anonymous referee for having shown them
the elementary proof of Proposition \ref{p:sisg-case-sing}.

The first author was supported in part by JSPS KAKENHI Grant Number JP18H01123.

The second author was supported in part by NSERC and the Canada Research Chairs program.





\end{document}